\documentclass{amsart}

\usepackage[latin1]{inputenc}

\usepackage{amsmath,amssymb,amscd,latexsym,color}
\usepackage{graphicx}

\usepackage[dvipdfm]{hyperref}

\usepackage[english]{babel}

\DeclareMathOperator{\Leb}{Leb}
\DeclareMathOperator{\Erg}{Erg}

\newcommand{\bor}[1][]{\mathcal{B}(#1)}

\newcommand{\N}{\mathbb N}

\newcommand{\Z}{\mathbb{Z}}

\newcommand{\Zd}{\mathbb{Z}^d}

\newcommand{\Q}{\mathbb{Q}}

\newcommand{\R}{\mathbb{R}}
\newcommand{\Rd}{\mathbb{R}^d}

\renewcommand{\P}{\mathbb{P}}
\newcommand{\E}{\mathbb{E}}

\newcommand{\Ed}{\mathbb{E}^d}

\newcommand{\Edo}{\overrightarrow{\mathbb{E}}^d}
\newcommand{\Ebarre}{\overline{\mathbb{E}}}
\newcommand{\Pbarre}{\overline{\mathbb{P}}}

\newcommand{\D}{\mathcal{D}}
\newcommand{\Eddo}{\overrightarrow{\mathbb{E}}^{d+1}_{\text{alt}}}
\renewcommand{\epsilon}{\varepsilon}
\renewcommand{\phi}{\varphi}
\renewcommand{\limsup}{\overline{\lim}}
\renewcommand{\liminf}{\underline{\lim}}
\newcommand{\pcdir}{\overrightarrow{p_c}^{\text{alt}}}
\newcommand{\Vddo}{{\mathbb{V}}^{d+1}}

\newcommand{\ie}{\emph{i.e. }}

\newcommand{\miniop}[3]{%
\renewcommand{\arraystretch}{0.6}
\begin{array}{c}
{\scriptstyle #1}\\
#2\\
{\scriptstyle #3}
\end{array}
\renewcommand{\arraystretch}{1}}

\newcommand{\Card}[1]{\vert #1 \vert}

\newcommand{\1}{1\hspace{-1.3mm}1}

\begin{document}

{
\newtheorem{theorem}{Theorem}[section]
\newtheorem{conjecture}[theorem]{Conjecture}

}
\newtheorem{lemme}[theorem]{Lemma}
\newtheorem{defi}[theorem]{Definition}
\newtheorem{coro}[theorem]{Corollary}
\newtheorem{rem}[theorem]{Remark}
\newtheorem{prop}[theorem]{Proposition}
\newtheorem*{hyp}{Assumptions}

\newcommand{\T}[2]{{#1}.{#2}} 

\title[Large deviations for the contact process]{Large deviations for the contact process in random environment}

{
\author{Olivier Garet}
\address{Institut \'Elie Cartan Nancy (math{\'e}matiques)\\
Universit{\'e} de Lorraine\\
Campus Scientifique, BP 239 \\
54506 Vandoeuvre-l{\`e}s-Nancy  Cedex France\\}
\email{Olivier.Garet@univ-lorraine.fr}
\author{R{\'e}gine Marchand}
\email{Regine.Marchand@univ-lorraine.fr}

}

\def\motsclefs{Random growth, contact process, random environment, shape theorem, large deviation inequalities}


\begin{abstract}
The asymptotic shape theorem for the contact process in random environment gives the existence of a norm $\mu$ on $\Rd$ such that the hitting time $t(x)$ is asymptotically equivalent to $\mu(x)$ when the contact process survives. We provide here exponential upper bounds for the probability of the event $\{\frac{t(x)}{\mu(x)}\not\in [1-\epsilon,1+\epsilon]\}$; these bounds are optimal for independent random environment. 
As a special case, this gives the large deviation inequality for the contact process in  a deterministic environment, which, as far as we know, has not been established yet.
\end{abstract}

{\maketitle
}
\setcounter{tocdepth}{1}
 
\section{Introduction}
Durrett and Griffeath~\cite{MR656515} proved that when the contact process on $\Zd$ starting from the origin survives, the set of sites occupied before time $t$ satisfies an asymptotic shape theorem, as in first-passage percolation. In~\cite{GM-contact}, we extended this result to the case of the contact process in a random environment. 

The random environment is given by a collection $(\lambda_e)_{e \in \Ed}$ of positive random variables indexed by the set of edges of the grid $\Zd$. Given a realization $\lambda$ of this environment, the contact process $(\xi^0_t)_{t\ge 0}$ in the environment $\lambda$ is a homogeneous Markov process taking its values in the set $\mathcal{P}(\Zd)$ of subsets of $\Zd$. If 
$\xi_t^0(z)=1$, we say that $z$ is occupied at time $t$, while if $\xi_t^0(z)=0$, we say that $z$ is empty at time~$t$.
The initial value of the process is $\{0\}$ and the process evolves as follows: 
\begin{itemize}
\item an occupied site becomes empty at rate $1$,
\item an empty site $z$ becomes occupied at rate:
$\displaystyle \sum_{\|z-z'\|_1=1} \xi_t^0(z')\lambda_{\{z,z'\}},$
\end{itemize}
all these evolutions being independent. We study then the hitting time $t(x)$ of a site $x$: 
$$t(x)=\inf\{ t\ge 0: \; x \in \xi^0_t \}.$$

In~\cite{GM-contact}, we proved that under good assumptions on the random environment, there exists a norm $\mu$ on $\Rd$ such that for almost every environment, the family $(t(x))_{x\in\Zd}$ satisfies, when $\|x\|$ goes to $+\infty$, 
$$t(x)\sim \mu(x) \quad \text{on the event ``the process survives''}.$$ 
We focus here on the large deviations of the hitting time $t(x)$ for the contact process in random environment. As far as we know, such inequalities for the classical contact process have not been studied yet, they will be contained in our results. 

The assumptions we will require on the random environment are the ones we already needed in~\cite{GM-contact}. We denote by $\lambda_c(\Zd)$ the critical intensity of the classical contact process on $\Zd$, we fix $\lambda_{\min}$ and $\lambda_{\max}$ such that
$\lambda_c(\Zd)<\lambda_{\min}\le \lambda_{\max}$ and we set $\Lambda=[\lambda_{\min},\lambda_{\max}]^{\Ed}$.
\begin{hyp}[E]
The support of the law $\nu$ of the random environment is included in $\Lambda=[\lambda_{\min},\lambda_{\max}]^{\Ed}$; the law $\nu$ is stationary, and if $\Erg(\nu)$ denotes the set of $x\in\Zd \backslash \{0\}$ such that the translation along vector $x$ is ergodic for $\nu$, then the cone generated by $\Erg(\nu)$ is dense in $\Rd$. 
\end{hyp}
This last condition is obviously fulfilled if $\Erg(\nu)=\Zd\backslash\{0\}$. We will sometimes require the stronger following assumptions:
\begin{hyp}[E'] The law $\nu$ of the random environment is a product measure: 
$\nu=\nu_0^{\otimes\Ed}$, where $\nu_0$ is some probability measure on $[\lambda_{\min},\lambda_{\max}]$.
\end{hyp}
By taking for $\nu$ the Dirac mass~$(\delta_{\lambda})^{\otimes \Ed}$, with $\lambda>\lambda_c(\Zd)$, which clearly fullfills these assumptions, we recover the case of the classical contact process in a deterministic environment. 

 For $\lambda \in \Lambda$, we denote by $\P_\lambda$ the (quenched) law of the contact process in environment $\lambda$, and by ${\Pbarre}_\lambda$ the (quenched) law of the contact process in environment $\lambda$ conditioned to survive. 
We define then the annealed probability measures $\Pbarre$ and $\P$:
$${\Pbarre}(.)=\int_\Lambda {\Pbarre}_\lambda(.)\ d\nu(\lambda)\quad \text{ and }\quad {\P}(.)=\int_\Lambda {\P}_\lambda(.)\ d\nu(\lambda).$$

We will study separately the probabilities of the ``upper large deviations'' and the ``lower large deviations'', \ie respectively of the events $\{t(x)\ge (1+\epsilon)\mu(x)\}$ and  $\{t(x)\le (1-\epsilon)\mu(x)\}$.

The most general result concerns the quenched ``upper large deviations'' for the hitting time $t(x)$ and the coupling time
$$t'(x)=\inf\{ T\ge 0:\;\forall t\ge T\quad \xi^0_t(x)= \xi^{\Zd}_t(x)\},$$ 
where $(\xi^{\Zd}_t)_{t \ge 0}$ is the contact process starting from $\Zd$, and for the set of hit points $H_t$ and the coupled region $K'_t$:
\begin{eqnarray*}
 H_t  =  \{x \in \Zd: \; t(x) \le t\}, &&\tilde{H}_t=H_t+[0,1]^d\\
K'_t  =  \{x\in\Zd: \; t'(x)\le t\}, &&\tilde{K}'_t=K'_t+[0,1]^d.
\end{eqnarray*}
We only require here  Assumptions~$(E)$.
\begin{theorem}
\label{theoGDUQ}
Let $\nu$ be an environment law satisfying Assumptions~$(E)$. \\
For every $\epsilon>0$, there exist $B>0$ and a random variable $A(\lambda)$ such that
for $\nu$ almost every environment $\lambda$, for every $x \in \Zd$,
\begin{eqnarray}
 \Pbarre_{\lambda}\left(t(x)\ge\mu(x)(1+\epsilon)\right) & \le & A(\lambda)e^{-B\|x\|}, \label{venus} \\
 \Pbarre_{\lambda}\left(t'(x)\ge\mu(x)(1+\epsilon)\right) & \le & A(\lambda)e^{-B\|x\|},  \label{tcouple} \\
 \Pbarre_{\lambda}\left(\forall t \ge T \quad (1-\epsilon)t A_\mu \subset \tilde{K'_t} \cap \tilde{H_t} \right) & \ge & 1-A(\lambda)e^{-BT}.  \label{audessousforme} 
\end{eqnarray}
\end{theorem}
We can note that the random variable $A(\lambda)$ is almost surely finite, but that it could often be large. This question will be studied  in a forecoming paper about annealed upper large deviations~\cite{GM-contact-gd-annealed}. The key point of the proof of Theorem~\ref{theoGDUQ}, interesting on its own, is to control the times $s$ when a site $x$ is occupied and has infinite progeny. We will denote this event by $\{(0,0) \to (x,s) \to \infty\}$ by analogy with percolation.
\begin{theorem}
\label{lemme-pointssourcescontact}
There exist $C, \theta,A,B>0$ such that $ \forall\lambda\in\Lambda \quad \forall x\in\Zd$
$$\forall t\ge C\|x\|\quad \Pbarre_{\lambda}\left(\Leb\{s \in [0,t]: (0,0)\to (x,s) \to \infty\} \le \theta t
\right)\le A\exp(-Bt). $$
\end{theorem}

For the ``lower large deviations'', the subadditivity gives a nice setting and allows to state a large deviations principle in the spirit of Hammersley~\cite{MR0370721}.
\begin{theorem}
\label{LDP}
Let $\nu$ be an environment law satisfying Assumptions~$(E)$.  \\
Let $x\in\Zd$. There exist a convex function $\Psi_x$ and a concave function $K_x$ taking their values in $\R_+$ such that for $\nu$ almost every $\lambda$,
\begin{eqnarray*}
 \forall u>0 \quad \lim_{n\to +\infty} -\frac1{n}\log \Pbarre_{\lambda}(t(nx)\le nu) & = & \Psi_x(u); \\
\forall\theta\ge 0\quad\lim_{n\to +\infty} -\frac1{n}\log \Ebarre_{\lambda}[e^{-\theta t(nx)} ]& = & K_x(\theta).
\end{eqnarray*}
The functions $\Psi_x$ and $K_x$ moreover satisfy the reciprocity relations:
$$\forall u>0 \quad \forall \theta\ge 0\quad \Psi_x(u)=\sup_{\theta\ge 0} \{K_x(\theta)-\theta u\}\text{ and }K_x(\theta)=\inf_{u>0} \{\Psi_x(u)+\theta u\}.$$ 
\end{theorem}
To obtain effective large deviation inequalities, we moreover have to prove that $\Psi_x(u)>0$ if  $u<\mu(x)$. More precisely, 
\begin{theorem}
\label{dessouscchouette}
Let $\nu$ be an environment law satisfying Assumptions~$(E')$. \\
For every $\epsilon>0$, there exist $A,B>0$ such that for every $x \in \Zd$, for every $t \ge 0$,
\begin{eqnarray}
\P (t(x)\le (1-\epsilon)\mu(x)) & \le & A\exp(-B\|x\|), \label{defontenay}\\
\P(\forall s \ge t \quad  H_s \subset (1+\epsilon)s A_\mu) & \ge & 1-A\exp(-Bt)\label{decadix}.
\end{eqnarray}
\end{theorem}
The annealed large deviations inequalities imply the quenched ones: setting
$$A(\lambda)=\sum_{x\in\Zd}\exp(B\|x\|/2)\P_{\lambda}\left(t(x)\le (1-\epsilon)\mu(x)\right),$$
we see that $A(\lambda)$ is integrable with respect to $\nu$, and thus is $\nu$-almost surely finite. So
$$\forall x\in\Zd\quad \P_{\lambda}(t(x)\le (1-\epsilon)\mu(x))\le A(\lambda)\exp(-B/2\|x\|).$$
Unfortunately, we do not have a complete large deviation principle as Theorem~\ref{LDP} for the upper large deviations. However, 
we will see in Section~\ref{bonnevitesse} that when the environment is i.i.d, the exponential order given by these inequalities is optimal.

Asymptotic shape results for growth models are generally proved thanks to the subadditive processes theory initiated by Hammersley and Welsh~\cite{MR0198576}, and especially with Kingman's  subadditive  ergodic theorem~\cite{MR0356192} and its extensions.
Since  Hammersley~\cite{MR0370721}, we know that subadditive properties offer a proper setting to study the large deviation inequalities. See also the survey by Grimmett~\cite{MR814710} and the Saint-Flour course by Kingman~\cite{MR0438477}. However, as noted by Sepp{\"a}l{\"a}inen and Yukich~\cite{MR1843178}, the general theory of large deviations for subadditive processes is patchy. The best known case is first-passage percolation, studied by Grimmett and Kesten in 1984~\cite{grimmett-kesten}. 
This paper introduced some lines of proof for the large deviations of growth processes, that have been reused later, for instance in  the study of the large deviations for the chemical distance in Bernoulli percolation~\cite{GM-large}. 
For more recent results concerning first-passage percolation, see Chow--Zhang~\cite{chow-zhang}, Cranston--Gauthier--Mountford~\cite{MR2521889}, and Théret et al~\cite{MR2464099,MR2343936,MR2610330,RT-IHP,CT-PTRF,CT-AAP,CT-TAM,RT-ESAIM}.

The renormalization techniques used by Grimmett and Kesten are well-known now: static renormalization for ``upper large deviations'' (control of a too slow growth), dynamic renormalization for ``lower large deviations'' (control of a too fast growth). However, the possibility for the contact process to die gives rise to extra difficulties that do not appear in the case of first-passage percolation or even of Bernoulli percolation. To our knowledge, the only growth process with possible extinction for which large deviations inequalities have been established is oriented percolation in dimension 2 (see Durrett~\cite{MR757768}). Note also that Proposition 20.1 in the PhD thesis of Couronné~\cite{Couronne} rules out the possibility of a too fast growth for oriented percolation in dimension $d$.

In Section 2, we construct the model, give the notation and state previous results, mainly from~\cite{GM-contact}. Section 3 is devoted to the proof of the upper large deviation inequalities, Theorem~\ref{theoGDUQ}, while lower large deviations -- Theorems~\ref{LDP} and~\ref{dessouscchouette} -- are proved in Section 4. Finally, the optimality of the exponential decrease given by these results is briefly discussed in Section 5.
\section{Preliminaries}
 
\subsection{Definition of the model}
Let $\lambda_{\min}$ and  $\lambda_{\max}$ be fixed such that
$\lambda_c(\Zd)<\lambda_{\min}\le\lambda_{\max}$, where
$\lambda_c(\Zd)$ is the critical parameter for the survival of the classical contact process on $\Zd$.
In the following, we restrict ourselves to the study of the contact process in random environment with birth rates  $\lambda=(\lambda_e)_{e \in \Ed}$ in $\Lambda=[\lambda_{\min},\lambda_{\max}]^{\Ed}$. An environment is thus a collection $\lambda=(\lambda_e)_{e \in \Ed} \in \Lambda$.

\medskip
Let $\lambda \in \Lambda$ be fixed. The contact process $(\xi_t)_{t\ge 0}$  in the environment $\lambda$ is a homogeneous Markov process taking its values in the set $\mathcal{P}(\Zd)$ of subsets of $\Zd$, that we sometimes identify with $\{0,1\}^{\Zd}$: for $z \in \Zd$ we also use the random variable $\xi_t(z)=\1_{\{z \in \xi_t\}}$.
If $\xi_t(z)=1$, we say that $z$ is occupied or infected, while if $\xi_t(z)=0$, we say that $z$ is empty or healthy. The evolution of the process is as follows:
\begin{itemize}
\item an occupied site becomes empty at rate $1$,
\item an empty site $z$ becomes occupied at rate
$\displaystyle \sum_{\|z-z'\|_1=1} \xi_t(z')\lambda_{\{z,z'\}},$
\end{itemize}
each of  these evolutions being independent from the others. In the following, we denote by $\D$ the set of càdlàg functions from $\R_{+}$ to $\mathcal{P}(\Zd)$: it is the set of trajectories for Markov processes with state space  $\mathcal{P}(\Zd)$.

To define the contact process in the environment $\lambda\in\Lambda$, we use  Harris' construction~\cite{MR0488377}. It allows to make a coupling between contact processes starting from distinct initial configurations by building them from a single collection of Poisson measures on~$\R_+$.

\subsubsection*{Graphical construction}
We endow $\R_+$ with the Borel $\sigma$-algebra $\mathcal B(\R_+)$, and we denote by $M$ the set of locally finite counting measures $m=\sum_{i=0}^{+\infty} \delta_{t_i}$. We endow this set with the $\sigma$-algebra $\mathcal M$ generated by the maps $m\mapsto m(B)$, where $B$ describes the set of Borel sets in $\R_+$.

We then define the measurable space $(\Omega, \mathcal F)$ by setting
$$\Omega=M^{\Ed}\times M^{\Zd} \text{ and } \mathcal F=\mathcal{M}^{\otimes \Ed} \otimes \mathcal{M}^{\otimes \Zd}.$$
On this space, we consider the family $(\P_{\lambda})_{\lambda\in\Lambda}$  of probability measures defined as follows: 
for every $\lambda=(\lambda_e)_{e \in \Ed} \in \Lambda$, 
$$\P_{\lambda}=\left(\bigotimes_{e \in \Ed} \mathcal{P}_{\lambda_{e}}\right) \otimes \mathcal{P}_1^{\otimes\Zd},$$
where, for every $\lambda\in\R_+$, $\mathcal{P}_{\lambda}$ is the law of a Poisson point process on $\R_+$ with intensity $\lambda$. If $\lambda \in \R_+$, we write $\P_\lambda$ (rather than $\P_{(\lambda)_{e \in \Ed}}$) for the law in deterministic environment with constant infection rate $\lambda$.

For every $t\ge 0$, we denote by $\mathcal{F}_t$ the $\sigma$-algebra generated by the maps $\omega\mapsto\omega_e(B)$ and $\omega\mapsto\omega_z(B)$, where $e$ ranges over all edges in  $\Ed$, $z$ ranges over all points in $\Zd$, and $B$  ranges over all Borel sets in $[0,t]$.

We build the contact process in environment $\lambda\in\Lambda$ from this family of Poisson process, as detailed in Harris~\cite{MR0488377} for the classical contact process and in~\cite{GM-contact} for the random environment case. Note especially that the process is attractive
$$(A \subset B)  \Rightarrow (\forall t \ge 0\quad \xi_t^A \subset \xi_t^B),$$
and Fellerian; then it enjoys the strong Markov property.

\subsubsection*{Time translations}
For $t \ge 0$, we define the translation operator $\theta_t$ on a locally finite counting measure $m=\sum_{i=1}^{+\infty} \delta_{t_i}$ on $\R_+$ by setting
$$\theta_t m=\sum_{i=1}^{+\infty} \1_{\{t_i\ge t\}}\delta_{t_i-t}.$$
The translation $\theta_t$ induces an operator on $\Omega$, still denoted by $\theta_t$: for every $\omega \in \Omega$, we set
$$ \theta_t \omega=((\theta_t \omega_e)_{e \in \Ed}, (\theta_t \omega_z)_{z \in \Zd}).$$

\subsubsection*{Spatial translations}
The group $\Zd$ can act on the process and on the environment. The action on the process changes the observer's point of view:
for   $x \in \Zd$, we define the translation operator~$T_x$ by 
$$\forall \omega \in \Omega\quad T_x \omega=(( \omega_{x+e})_{e \in \Ed}, ( \omega_{x+z})_{z \in \Zd}),$$
where $x+e$ the edge $e$ translated by vector $x$.

Besides, we can consider the translated environment $\T{x}{\lambda}$ defined by $(\T{x}{\lambda})_e=\lambda_{x+e}$.
These actions are dual in the sense that for every $\lambda \in \Lambda$, for every  $x \in \Zd$, 
\begin{eqnarray}
\label{translationspatiale}
\forall A\in\mathcal{F}\quad\P_{\lambda}(T_x \omega \in A) & = & \P_{\T{x}{\lambda}}(\omega \in A).
\end{eqnarray}
Consequently, the law of $\xi^x$ under $\P_\lambda$ coincides with the law of $\xi^0$ under $\P_{x.\lambda}$.

\subsubsection*{Essential hitting times and associated translations}

For a set $A \subset \Zd$, we define the lifetime $\tau^A$ of the process starting from $A$ by 
$$\tau^A=\inf\{t\ge0: \; \xi_t^A=\varnothing\}. $$
For $A \subset \Zd$ and $x \in \Zd$,  we also define the first infection time $t^A(x)$ of the site $x$ from the set $A$ by 
$$t^A(x)=\inf\{t\ge 0: \; x \in \xi_t^A\}.$$
If $y\in\Zd$, we write  $t^y(x)$ instead of  $t^{\{y\}}(x)$. Similarly, we simply write $t(x)$ for $t^0(x)$.

In our previous paper~\cite{GM-contact}, we introduced a new quantity $\sigma(x)$:  it is a time when the site $x$ is infected from the origin $0$ and also has an infinite lifetime. This essential hitting time is defined from a family of stopping times as follows: we set $u_0(x)=v_0(x)=0$ and we define recursively two increasing sequences of stopping times $(u_n(x))_{n \ge 0}$ and $(v_n(x))_{n \ge 0}$ with
$u_0(x)=v_0(x)\le u_1(x)\le v_1(x)\le u_2(x)\dots$ as follows:
\begin{itemize}
\item Assume that $v_k(x)$ is defined. We set $u_{k+1}(x)  =\inf\{t\ge v_k(x): \; x \in \xi^0_t \}$. \\
If $v_k(x)<+\infty$, then $u_{k+1}(x)$ is the first time after $v_k(x)$ where site $x$ is once again infected; otherwise, $u_{k+1}(x)=+\infty$.
\item Assume that $u_k(x)$ is defined, with $k \ge 1$. We set $v_k(x)=u_k(x)+\tau^x\circ \theta_{u_k(x)}$.\\
If $u_k(x)<+\infty$, the time $\tau^x\circ \theta_{u_k(x)}$ is the lifetime of the contact process starting from $x$ at time $u_k(x)$; otherwise, $v_k(x)=+\infty$.
\end{itemize}
We then set
\begin{equation}
\label{definitiondeK}
K(x)=\min\{n\ge 0: \; v_{n}(x)=+\infty \text{ or } u_{n+1}(x)=+\infty\}.
\end{equation}
This quantity represents the number of steps before the success of this process: either we stop because we have just found an infinite $v_n(x)$, which corresponds to a time $u_n(x)$ when $x$ is occupied and has infinite progeny, or we stop because we have just found an infinite $u_{n+1}(x)$, which says that after $v_n(x)$, site $x$ is nevermore infected.

We proved that $K(x)$ is almost surely finite, which allows to define the essential hitting time $\sigma(x)$ by setting $\sigma(x)=u_{K(x)}$.
It is of course larger than the hitting time $t(x)$ and can been seen as a regeneration time.

Note however that $\sigma(x)$ is not necessary the first time when $x$ is occupied and has infinite progeny: for instance, such an event can occur between $u_1(x)$ and $v_1(x)$, being ignored by the recursive construction.

At the same time, we define the operator $\tilde \theta_x$ on $\Omega$ by:
\begin{equation*}
\tilde \theta_x = 
\begin{cases} T_{x} \circ \theta_{\sigma(x)} & \text{if $\sigma(x)<+\infty$,}
\\
T_x &\text{otherwise,}
\end{cases}
\end{equation*}
or, more explicitly,
\begin{equation*}
(\tilde \theta_x)(\omega) = 
\begin{cases} T_{x} (\theta_{\sigma(x)(\omega)} \omega) & \text{if $\sigma(x)(\omega)<+\infty$,}
\\
T_x (\omega) &\text{otherwise.}
\end{cases}
\end{equation*}
We will mainly deal with the essential hitting time $\sigma(x)$ that enjoys, unlike $t(x)$, 
some good invariance properties in the survival-conditioned environment.  Moreover, the difference between $\sigma(x)$ and $t(x)$ was controlled in~\cite{GM-contact}; this will allow us to transpose to $t(x)$ the results obtained for $\sigma(x)$.

\subsubsection*{Contact process in the survival-conditioned environment}
For $\lambda \in \Lambda$, we define the probability measure ${\Pbarre}_\lambda$ 
on $(\Omega, \mathcal F)$ by
$$\forall E\in\mathcal{F}\quad {\Pbarre}_\lambda(E)=\P_\lambda(E|\tau^0=+\infty).$$
It is thus the law of the family of Poisson point processes, conditioned to the survival of the contact process starting from $0$.
Let then $\nu$ be a probability measure on the set of environments $\Lambda$.
On the same space $(\Omega, \mathcal F)$, we define the corresponding annealed probabilities $\Pbarre$ and $\P$ by setting
$$\forall E\in\mathcal{F}\quad {\Pbarre}(E)=\int_\Lambda {\Pbarre}_\lambda(E)\ d\nu(\lambda) \quad{ and } \quad {\P}(E)=\int_\Lambda {\P}_\lambda(E)\ d\nu(\lambda).$$

\subsection{Previous results} We recall here the results established in~\cite{GM-contact} for the contact process in random environment.

\begin{prop}[Lemma 8 and Corollary 9 in~\cite{GM-contact}]
 \label{magic}
Let $x,y \in \Zd \backslash \{0\}$, $\lambda\in\Lambda$, $A$ in the $\sigma$-algebra generated by $\sigma(x)$, and $B\in \mathcal F$. Then
$$\forall \lambda \in \Lambda \quad \Pbarre_\lambda(A \cap (\tilde{\theta}_x)^{-1}(B))=\Pbarre_\lambda(A) \Pbarre_{\T{x}{\lambda}}(B).$$
\label{invariancePbarre}
As consequences we have:
\begin{itemize}
\item The probability measure $\Pbarre$ is invariant under the translation $\tilde \theta_x$.
\item Under $\Pbarre_\lambda$, $\sigma(y)\circ\tilde{\theta}_x$ and $\sigma(x)$ are independent. Moreover, the law of $\sigma(y)\circ\tilde{\theta}_x$ under $\Pbarre_\lambda$ is the same as the law of  $\sigma(y)$ under $\Pbarre_{\T{x}{\lambda}}$.
\item  The random variables $(\sigma(x) \circ (\tilde \theta_{x})^j)_{j \ge 0}$ are independent under~$\Pbarre_\lambda$. 
\end{itemize}
\end{prop}

\begin{prop}[Corollaries~20 and 21 in~\cite{GM-contact}]
\label{propmoments}
There exist $A,B,C>0$ and, for every $p\ge 1$, a constant $C_p>0$ such that for every $x\in\Zd$ and every $\lambda\in\Lambda$, 
\begin{eqnarray}\label{moms}
\Ebarre_{\lambda} [\sigma(x)^p ]& \le&  C_p (1+\|x\|)^{p},\\
\label{asigma}
\forall t\ge 0 \quad ( \|x\|\le  t) & \Longrightarrow & \left(\Pbarre_{\lambda}(\sigma(x)> Ct)  \le  A\exp(-Bt^{1/2})\right).
\end{eqnarray}
\end{prop}

\begin{prop}[Theorem 2 in~\cite{GM-contact}]
\label{systemeergodique}
For every $x\in\Erg(\nu)$, the measure-preserving dynamical system $(\Omega,\mathcal{F},\Pbarre,\tilde{\theta}_x)$ is ergodic.
\end{prop}

We then proved that $\Pbarre$ almost surely, for every $x \in \Zd$, $\frac{\sigma(nx)}n$ converges to a deterministic real number $\mu(x)$. 
The function $x\mapsto \mu(x)$ can be extended to a norm on $\Rd$, that characterizes the asymptotic shape. Let $A_{\mu}$ be the unit ball for $\mu$.
We define
\begin{eqnarray*}
H_t & = & \{x\in\Zd: \; t(x)\le t\},\\
G_t & = & \{x\in\Zd: \; \sigma(x)\le t\},\\
K'_t & = & \{x\in\Zd: \;\forall s\ge t \quad \xi^0_s(x)=\xi^{\Zd}_s(x)\},
\end{eqnarray*}
and we denote by $\tilde{H}_t,\tilde{G}_t,\tilde{K}'_t$ their "fattened" versions: 
$$\tilde{H}_t=H_t+[0,1]^d, \; \tilde{G}_t=G_t+[0,1]^d \text{ and } \tilde{K}'_t=K'_t+[0,1]^d.$$
We can now state the asymptotic shape result:

\begin{prop}[Theorem 3 in~\cite{GM-contact}]
\label{thFA}
For every $\epsilon>0$, $\Pbarre-a.s.$, for every~$t$ large enough,
\begin{equation}
\label{leqdeforme}
(1-\epsilon)A_{\mu}\subset \frac{\tilde K'_t\cap \tilde G_t}t\subset \frac{\tilde G_t}t\subset\frac{\tilde H_t}t\subset (1+\epsilon)A_{\mu}.
\end{equation}
\end{prop}

In order to prove the asymptotic shape theorem, we established 
exponential controls uniform in $\lambda \in \Lambda$. We set
$$B_r^x=\{y \in  \Zd: \; \|y-x\|_{\infty} \le r\},$$
and we write $B_r$ instead of  $B_r^0$. 

\begin{prop}[Proposition 5 in~\cite{GM-contact}]
\label{propuniforme}
There exist $A,B,M,c,\rho>0$ such that for every
$\lambda\in\Lambda$, for every  $y \in \Zd$, for every  $ t\ge0$
\begin{eqnarray}
\P_\lambda(\tau^0=+\infty) & \ge & \rho,
\label{uniftau} \\
\P_\lambda(H^0_t \not\subset B_{Mt} ) & \le & A\exp(-Bt), 
\label{richard} \\
\P_\lambda ( t<\tau^0<+\infty) &\le&  A\exp(-Bt), \label{grosamasfinis} \\
 \P_{\lambda}\left( t^0(y)\ge \frac{\|y\|}c+t,\; \tau^0=+\infty \right) & \le & A\exp(-Bt),
\label{retouche}\\
\P_{\lambda}(0\not\in K'_t, \; \tau^0=+\infty) &\le &A\exp(-B t).
\label{petitsouscouple} 
\end{eqnarray}
\end{prop}

\begin{lemme}
\label{momtprime}
There exist $A,B,C>0$ such that for every $x\in\Zd$ and every $\lambda\in\Lambda$, 
\begin{equation}
\label{momtprimeeq}
\forall t\ge 0 \quad ( \|x\|\le  t)  \Longrightarrow  \left(\Pbarre_{\lambda}(t'(x)> Ct)  \le  A\exp(-Bt^{1/2})\right).
\end{equation}
\end{lemme}

\begin{proof}
For every $\lambda \in \Lambda$, for every $ x \in \Zd$,
\begin{eqnarray}
\Pbarre_{\lambda}(t'(x)>\sigma(x)+s) & = & \Pbarre_{\lambda}(x \not\in K'_{\sigma(x)+s}\cap G_{\sigma(x)+s})\nonumber \\
 & = & \Pbarre_{\lambda}(x \not\in K'_{\sigma(x)+s})\nonumber \\
& \le &  \Pbarre_{\lambda}( x \not\in x+(K'_s) \circ \tilde{\theta}_x)= \Pbarre_{x.\lambda}(0 \not\in K'_s\nonumber)\\
& \le & A \exp(-Bs),\label{demai}
\end{eqnarray}
with~(\ref{uniftau}) and~(\ref{petitsouscouple}).
With~(\ref{asigma}), this estimate gives the announced result. 
\end{proof}

\subsection{An abstract restart procedure}
We formalize here the restart procedure for Markov chains.
Let $E$ be the state space where our Markov chains $(X^x_n)_{n\ge 0}$ evolve,  $x \in E$ being the starting point of the chain.
We  suppose that we have on our disposal a set  $\tilde{\Omega}$, an update function $f:E\times \tilde{\Omega}\to E$, and a probability measure~$\nu$ on $\tilde{\Omega}$ such that on the probability space 
$(\Omega, \mathcal{F}, \P)=(\tilde{\Omega}^{\N^*},\bor[\tilde{\Omega}^{\N^*}],\nu^{\otimes\N^*})$, endowed with the natural filtering $(\mathcal{F}_n)_{n\ge 0}$ given by $\mathcal{F}_n=\sigma(\omega\mapsto \omega_k: \;k\le n)$, the chains $(X^x_n)_{n\ge 0}$ starting from the different states enjoy the following representation: 
\begin{eqnarray*}
\begin{cases}
X^x_0(\omega)=x \\
X^x_{n+1}(\omega)=f(X^x_n(\omega),\omega_{n+1}).	
	\end{cases}
\end{eqnarray*}
As usual, we define $\theta:\Omega\to\Omega$ which maps $\omega=(\omega_n)_{n\ge 1}$ to $\theta\omega=(\omega_{n+1})_{n\ge 1}$.
We assume that for each $x\in E$, we have defined a $(\mathcal{F}_n)_{n\ge 0}$-adapted stopping time~$T^x$, a $\mathcal{F}_{T^x}$-measurable function $G^x$ and a $\mathcal{F}$-measurable  function $F^x$.
Now, we are interested in the following quantities:
\begin{eqnarray*}
T_0^x=0 \text{ and } T^x_{k+1} & = & 
	\begin{cases}
	+\infty & \text{if }T^x_{k}=+\infty\\
	T_k^x+T^{x_k}(\theta_{T_k^{x}}) & \text{with $x_k=X^x_{\theta_{T_k^x}}$ otherwise;}
	\end{cases} \\
K^x & = & \inf\{k\ge 0:\;T_{k+1}^x=+\infty\}; \\
M^x & = & \sum_{k=0}^{K^x-1}  G^{x_k}(\theta_{T_k^x})+F^{X^{x_K}}(\theta_{T^x_{K}}).
\end{eqnarray*}
We wish to control the exponential moments of the $M^x$'s with the help of
exponential bounds for  $G^x$ and $F^x$.
In numerous applications to directed percolation or to the contact process, $T^x$ is the extinction time of the process (or of some embedded process) starting from the smallest point (in lexicographic order) in the configuration $x$.

\begin{lemme}[Lemma 4.1 in~\cite{GM-dop}]
\label{restartabstrait}
We suppose that there exist real numbers $A>0$, $c<1$, $p>0$, $\beta>0$ such that the real-valued functions $(G^x)_{x\in E},(F^x)_{x\in E}$  defined above satisfy $$\forall x\in E\quad  \left\lbrace
\begin{array}{l}
\mathbf{G}(x)=\E [\exp(\beta G^x)\1_{\{T^x<+\infty\}}]\le  c;\\
\mathbf{F}(x)=\E [\1_{\{T^x=+\infty\}} \exp(\beta F^x)]\le A;\\
\mathbf{T}(x)=\P (T^x=+\infty)\ge p.
\end{array}
\right.
$$
Then, for each $x \in E$, $K^x$ is $\P$-almost surely finite and 
$$ \E[ \exp(\beta M^x)]\le \frac{A}{1-c} <+\infty.$$
\end{lemme}

\subsection{Oriented percolation}

We work, for $d \ge1$, on the following graph: 
\begin{itemize}
\item The set of sites is $\Vddo=\{(z,n)\in \Zd \times \N\}$.
\item We put an oriented edge from $(z_1,n_1)$ to $(z_2,n_2)$ if and only if $n_2=n_1+1$ and $\|z_2-z_1\|_1\le1$; the set of these edges is denoted by $\Eddo$. 
\end{itemize}
Define $\Edo$ in the following way: in $\Edo$, there is an oriented edge between two points $z_1$ and $z_2$ in $\Zd$ if and only if  $\|z_1-z_2\|_1\le 1$.
The oriented edge in $\Eddo$  from $(z_1,n_1)$ to $(z_2,n_2)$ can be identified with the couple $((z_1,z_2),n_2)\in\Edo\times\N^*$. Thus, we identify $\Eddo$ and $\Edo\times\N^*$.

We consider $\Omega=\{0,1\}^{\Eddo}$ endowed with its Borel $\sigma$-algebra: the edges $e$ such that $\omega_e=1$ are said to be open, the other ones are closed. For $v, w$ in $\Zd\times\N$, we denote by $v \to w$ the existence of an oriented path from $v$ to $w$ composed of open edges. We denote by $\pcdir(d+1)$ the critical parameter for the Bernoulli oriented percolation on this graph (\ie all edges are independently open with probability~$p$).
We set, for $n \in \N$ and $(x,0)\in \Vddo$,
\begin{eqnarray*}
\bar{\xi}^x_n & = & \{y \in \Z:  \; (x,0)\to(y,n)\}, \\
\bar{\tau}^x & = & \max\{n \in \N:\; \bar{\xi}^x_n \neq \varnothing\}.
\end{eqnarray*}

We recall results from~\cite{GM-dop} for a class $\mathcal{C}_d(M,q)$ of dependent oriented percolation models on this graph. The parameter $M$ controls the range of the dependence while the parameter $q$ controls the probability for an edge to be open.
\begin{defi}[Class  $\mathcal{C}_d(M,q)$]
Let $d\ge1$ be fixed. Let $M$ be a positive integer and $q\in (0,1)$.

Let $(\Omega,\mathcal{F},\P)$ be a probability space endowed with a filtration  $(\mathcal{G}_n)_{n\ge 0}$. We assume that, on this probability space, a random field $(W^n_e)_{e\in\Edo,n\ge 1}$ taking its values in $\{0,1\}$ is defined. This field gives the states -- open or closed -- of the edges in $\Eddo$. 
We say that the law of the field
$(W^n_e)_{e\in\Edo,n \ge 1}$ is in $\mathcal{C}_d(M,q)$ if it satisfies the two following conditions. 
\begin{itemize}
\item  $\forall n\ge 1,\forall e \in \Edo\quad W^n_e\in\mathcal{G}_n$;
\item $\forall n \ge 0,\forall e \in \Edo\quad \P[W^{n+1}_e=1|\mathcal{G}_n\vee \sigma(W^{n+1}_f, \; d(e,f)\ge M)]\ge q$,
\end{itemize}
where $\sigma(W^{n+1}_f, \; d(e,f)\ge M)$ is the $\sigma$-field generated by the random variables $W^{n+1}_f$, with $d(e,f)\ge M$.
\end{defi}
Note that if $0\le q\le q'\le 1$, we have $\mathcal{C}_d(M,q')\subset \mathcal{C}_d(M,q)$.

We can control the probability of survival and also the lifetime  for these dependent oriented percolations.

\begin{prop}[Corollary 3.1 in~\cite{GM-dop}]
\label{petitmomentexpo}
Let $\epsilon>0$ and $M>1$. There exist $\beta>0$ and $q<1$ such that for each $\chi\in\mathcal{C}_d(M,q)$, 
$$\forall x \in \Zd \quad \E_{\chi}[\1_{\{\bar{\tau}^x<+\infty\}}\exp(\beta\bar{\tau}^x)]\le \epsilon \quad  \text{ and } \quad \chi(\bar{\tau}^x=+\infty) \ge 1-\epsilon.$$
\end{prop}

A point $(y,k) \in \Zd \times \N$ such that $(x,0)\to(y,k)\to\infty$ is called an immortal descendant of $x$. We will need estimates on the density of immortal descendants of $x$ above some given point $y$ in oriented dependent percolation.
So we define
\begin{eqnarray*}
\bar{G}(x,y)&=& \{k\in\N\quad (x,0)\to(y,k)\to\infty\}, \\
\bar{\gamma}(\theta,x,y) & = & \inf\{n\in\N: \quad \forall k\ge n \quad \Card{\{0,\dots,k\}\cap \bar{G}(x,y)}\ge\theta k\}.
\end{eqnarray*} 

\begin{prop}[Corollary 3.3 in~\cite{GM-dop}]
\label{lineairegamma}
Let 
$M>1$. There exist $q_0<1$ and positive constants $A,B,\theta,\alpha$  such that for each  $\chi\in\mathcal{C}_d(M,q_0)$, we have
$$\forall x,y\in\Zd\quad\forall n\ge 0\quad \chi(+\infty>\gamma(\theta,x,y)> \alpha \|x-y\|_1+n)\le Ae^{-Bn}.$$
\end{prop}

\section{Quenched upper large deviations}

The aim is now to prove the quenched upper large deviations of Theorem~\ref{theoGDUQ}.
In order to exploit the subadditivity, we show that $\sigma(x)$ admits exponential moments uniformly in $\lambda \in \Lambda$:
\begin{theorem}
\label{theomomexpsigma}
There exist positive constants $\gamma_1,\beta_1$ such that 
\begin{equation}
\label{momexpsigma}
\forall x \in \Zd \quad \forall \lambda \in \Lambda \quad   
\Ebarre_{\lambda}(e^{\gamma_1 \sigma(x)})  \le  e^{\beta_1\|x\|_{1}}.
\end{equation}
\end{theorem}

As an immediate consequence, we get

\begin{coro}
\label{sauveur}
There exist positive constants $A,B,c,$ such that for each
$\lambda\in\Lambda$, each  $x \in \Zd$ and every $ t\ge0$
\begin{eqnarray*}
 \Pbarre_{\lambda}\left( t'(x)\ge \frac{\|x\|}c+t \right) & \le & A\exp(-Bt).
\end{eqnarray*}
\end{coro}
\begin{proof}
$$\Pbarre_{\lambda}\left( t'(x)\ge \frac{\|x\|}c+t\right)\le \Pbarre_{\lambda}\left( \sigma(x)\ge \frac{\|x\|}c+t/2\right)+\Pbarre_{\lambda}(t'(x)-\sigma(x)\ge t/2).$$
The second term is controlled by Inequality~\eqref{demai} and Theorem~\ref{theomomexpsigma} gives the desired result with $c=\frac{\gamma_1}{\beta_1}$.
\end{proof}

The rest of this section is organized as follows. We first prove how the subadditive properties and the existence of exponential moments for $\sigma$ given by Theorem~\ref{theomomexpsigma}  imply the large deviations inequalities of Theorem~\ref{theoGDUQ}. Next we show how Theorem~\ref{lemme-pointssourcescontact} gives Theorem~\ref{theomomexpsigma}. Finally, the last (and most important) part will be devoted to the proof of Theorem~\ref{lemme-pointssourcescontact}.

\subsection{Proof of Theorem~\ref{theoGDUQ} from Theorem~\ref{theomomexpsigma}}

Let $\epsilon>0$. Let $\beta_1$ and $\gamma_1$ be the  constants given by~(\ref{momexpsigma}), and let
\begin{equation}
C>2 \beta_1/\gamma_1. \label{jechoisisC}
\end{equation}

Theorem~\ref{thFA} gives the almost sure convergence of $\sigma(x)/\mu(x)$ to $1$ when $\|x\|$ tends to $+\infty$, and Proposition~\ref{propmoments} ensures that the family  $(\sigma(x)/\mu(x))_{x \in \Zd}$ is bounded in $L^2(\Pbarre)$, therefore uniformly integrable: then the convergence also holds in $L^1(\Pbarre)$. 

Let then $M_0$ be such that
\begin{equation}
\label{je choisisM0}
( \mu(x) \ge M_0) \quad \Rightarrow \quad \left( \frac{\Ebarre(\sigma(x))}{\mu(x)}\right) \le 1+\epsilon/8.
\end{equation}
We assumed that $\{ay: \;a\in\R_+,y\in\Erg(\nu)\}$ is dense in $\Rd$.
Its range by $x\mapsto \frac{x}{\mu(x)}$ is therefore dense in $\{x\in\Rd: \;\mu(x)=1\}$, thus 
the set
$\{\frac{y}{\mu(y)}: \;y\in\Erg(\nu), \,\mu(y)\ge M_0\}$ is also dense in
 $\{x\in\Rd:\;\mu(x)=1\}$.
By a compactness argument, one can find a finite subset $F$
in $\{\frac{y}{\mu(y)}:\;y\in\Erg(\nu), \,\mu(y)\ge M_0\}$
such that
$$\forall \hat{x}\in \Rd \text{ such that } \mu(\hat{x})=1 \quad\exists  y\in F, \; \left\|\frac{y}{\mu(y)}-\hat{x}\right\|_1\le \epsilon/C.$$
We let $M=\max\{\mu(y):\;y\in F\}.$

For $y\in F$, note $\tilde{\sigma}(y)=\sigma(y)-(1+\frac{\epsilon}4)\mu(y)$. Since, with (\ref{momexpsigma}),  $\tilde{\sigma}(y)$ admits exponential moments, the asymptotics $\Ebarre[e^{t\tilde{\sigma}(y)}]=1+t\Ebarre[\tilde{\sigma}(y)]+o(t)$ holds in the neighborood of $0$.
Since $\Ebarre[\tilde{\sigma}(y)]<0$, we have $\Ebarre[e^{t\tilde{\sigma}(y)}]<1$ when $t$ is  small enough. Since $F$ is finite, 
we can find some constants $\alpha>0$ and $c_\alpha<1$ such that
\begin{eqnarray*}
 \forall y\in F && \Ebarre[ \exp(\alpha (\sigma(y)-(1+\epsilon/4)))]\le c_{\alpha}.
\end{eqnarray*}
Let $x\in\Zd$. We  associate to $x$ a point $y\in F$ and an integer $n$ such that
\begin{equation}
\label{jechoisisxn}
\left\| \frac{x}{\mu(x)}-\frac{y}{\mu(y)}\right\|_1\le \frac{\epsilon}C \text{ and } \left|n-\frac{\mu(x)}{\mu(y)}\right|\le 1.
\end{equation}
By the definition of $t(x)$, for each $\lambda \in \Lambda$, we have
\begin{eqnarray}
&& \Pbarre_{\lambda} \left( t(x) \ge (1+\epsilon)\mu(x) \right) \nonumber \\
& \le & \Pbarre_{\lambda} \left( \sum_{i=0}^{n-1} \sigma(y)\circ\tilde{\theta}_y^i + \sigma(x-ny)\circ\tilde{\theta}_y^n \ge (1+\epsilon)\mu(x) \right) \nonumber \\
& \le & \Pbarre_{\lambda} \left( \sum_{i=0}^{n-1} \sigma(y)\circ\tilde{\theta}_y^i \ge \left(1+\frac{\epsilon}2 \right)\mu(x)\right) + \Pbarre_{\lambda} \left(\sigma(x-ny)\circ\tilde{\theta}_y^n \ge \frac{\epsilon}2 \mu(x) \right).
\label{deuxtermes}
\end{eqnarray}
Let first consider the first term in~(\ref{deuxtermes}). With Proposition~\ref{invariancePbarre} and estimate~(\ref{momexpsigma}), it follows that
\begin{eqnarray*}
\Pbarre_{\lambda} \left(\sigma(x-ny)\circ\tilde{\theta}_y^n \ge \frac{\epsilon}2 \mu(x) \right)
& = & \Pbarre_{ny. \lambda} \left(\sigma(x-ny) \ge \frac{\epsilon}2 \mu(x) \right) \\
& \le & \exp\left( -\frac{\gamma_1 \epsilon \mu(x)}2 \right) \Ebarre_{ny.\lambda}( \exp(\gamma_1 \sigma(x-ny)))\\
& \le & \exp\left( -\frac{\gamma_1 \epsilon \mu(x)}2 \right)  \exp(\beta_1 \|x-ny\|_1).
\end{eqnarray*}
Our choices~(\ref{jechoisisxn}) for $y$ and $n$ and the definition of $M$ ensure that
$$\|x-ny\|_1\le \left\| x -\frac{\mu(x)}{\mu(y)} y \right\|_1+\left| \frac{\mu(x)}{\mu(y)} - n \right| \|y\|_1 \le \frac{\epsilon \mu(x)}C+M.$$
Our choice~(\ref{jechoisisC}) for $C$ gives then the existence of two positive constants $A_1$ and $B_1$ such that for each $\lambda \in \Lambda$ and each $x \in \Zd$,
$$\Pbarre_{\lambda}\left(\sigma(x-ny)\ge \frac{\epsilon}{2}\mu(x)\right)\le A_1\exp(-B_1\|x\|).$$
Let us move to the first term of~(\ref{deuxtermes}). Our choices~(\ref{jechoisisxn}) for $y$ and $n$ ensure that 
$$ \left| \frac{\mu(x)}{n \mu(y)}-1\right| \le \frac1n \le \left(\frac{\mu(x)}{M}-1\right)^{-1}.$$
Then, we can find $T$ sufficiently large to have, for $\mu(x) \ge T$, that
$$\frac{\mu(x)}{n \mu(y)} \ge \frac{1+\epsilon/4}{1+\epsilon/2}.$$

Suppose now that $\mu(x) \ge T$. Proposition~\ref{invariancePbarre} ensures that the variables $\sigma(y)\circ\tilde{\theta}_y^i$ are  independent under $\Pbarre_\lambda$ and moreover that the law of $\sigma(y)\circ\tilde{\theta}_y^i$ under $\Pbarre_\lambda$ coincides with the law of $\sigma(y)$ under $\Pbarre_{iy.\lambda}$ : thus
\begin{eqnarray*}
&&\Pbarre_{\lambda} \left( \sum_{i=0}^{n-1} \sigma(y)\circ\tilde{\theta}_y^i \ge \left(1+\frac{\epsilon}2 \right)\mu(x)\right)\\& \le & \Pbarre_{\lambda} \left( \sum_{i=0}^{n-1} \sigma(y)\circ\tilde{\theta}_y^i\ge (1+\frac{\epsilon}4) n\mu(y)\right) \\
&\le & \Pbarre_{\lambda} \left( \prod_{i=0}^{n-1} \exp \left( \alpha [\sigma(y)\circ\tilde{\theta}_y^i - (1+\frac{\epsilon}4)\mu(y)]\right) \ge 1 \right)  \\
& \le & \prod_{i=0}^{n-1} \Ebarre_{iy.\lambda} \left[ \exp \left( \alpha(\sigma(y)-(1+\frac{\epsilon}4)\mu(y)) \right) \right].
\end{eqnarray*}
Applying the Ergodic Theorem to the system $(\Lambda,\mathcal{B}(\Lambda),\nu, y.)$ and to the function $\lambda \mapsto \log \Ebarre_{\lambda} \left( \exp[\alpha(\sigma(y)-(1+\epsilon/4)\mu(y))] \right)$, we get that for $\nu$-almost every $\lambda$ and for each $y\in F$,
\begin{eqnarray*}
&& \miniop{}{\limsup}{n\to+\infty} \frac1{n}\log \Pbarre_{\lambda} \left( \frac{1}{n\mu(y)} \sum_{i=0}^{n-1} \sigma(y)\circ\tilde{\theta}_y^i \ge 1+\epsilon/4 \right) \\
& \le & \int_{\Lambda} \log \Ebarre_{\lambda} \left(\exp[\alpha(\sigma(y)-(1+\epsilon/4)\mu(y))] \right) d\nu(\lambda)\\
& \le & \log \int_{\Lambda}  \Ebarre_{\lambda} \left( \exp[\alpha(\sigma(y)-(1+\epsilon/4)\mu(y))]\right) d\nu(\lambda) \le \log c_{\alpha}<0.
\end{eqnarray*}
Using the norm equivalence theorem and noting that the choices~(\ref{jechoisisxn}) for $n$ and $y$ ensure that
$$ \frac{n}{\mu(x)} \le \frac{1}{M} +\frac{1}{T},$$
we deduce that
$$ \miniop{}{\limsup}{\|x\|\to +\infty}\frac{\log \Pbarre_{\lambda}(t(x)\ge\mu(x)(1+\epsilon))}{\|x\|}\le -C_{\epsilon},
$$ 
with $C_{\epsilon}=\max(-\log c_{\alpha},B_1)$. 
Inequality~(\ref{venus}) of Theorem~\ref{theoGDUQ} follows (with another  $C_{\epsilon}$, if necessary).

Let us move to the proof of inequality~(\ref{tcouple}) of Theorem~\ref{theoGDUQ}.
Let $T= \sum_{i=0}^{n-1} \sigma(y)\circ\tilde{\theta}_y^i + \sigma(x-ny)\circ\tilde{\theta}_y^n$. Using Corollary~\ref{invariancePbarre} repeatedly, the same reasoning as in the proof of Lemma~\ref{momtprime} gives $\Pbarre_{\lambda}(t'(x)>T+\epsilon\mu(x))\le \Pbarre_{x.\lambda}(0\not\in K'_{\epsilon\mu(x)})\le A\exp(-B\mu(x))$. Thus,  since $\Pbarre_{\lambda}(t'(x)>(1+2\epsilon)\mu(x))\le \Pbarre_{\lambda}(T>(1+\epsilon)\mu(x))+\Pbarre_{\lambda}(t'(x)>T+\epsilon\mu(x))$ and  $T$ has already been controlled, inequality~\eqref{tcouple} follows.

Let us prove inequality~\eqref{audessousforme} of Theorem~\ref{theoGDUQ}.
Since $t\mapsto K'_t\cap H_t$ is non-decreasing, it is sufficient to prove that there exist constants $A,B>0$  such thay
$$\forall n\in\N\quad \Pbarre((1-\epsilon)nA_{\mu}\not\subset K'_n\cap H_n)\le A\exp(-Bn).$$
The proof of the last inequality is classic. For points that have a small norm, we use inequality~\eqref{retouche} and Corollary~\ref{sauveur}; for the other ones, we use inequalities~\eqref{venus} and~\eqref{tcouple}.  

\subsection{Proof of Theorem~\ref{theomomexpsigma} from Theorem~\ref{lemme-pointssourcescontact}}

Theorem~\ref{lemme-pointssourcescontact} ensures that with a probability exceeding $1-A\exp(-Bt)$, the Lebesgue measure of the times $s\le C\|x\|+t$ when $(0,0) \to (x,s) \to \infty$ is at least $\theta t$. If $\sigma(x)\ge C\|x\|_\infty+t$, it means that all these times are ignored by the recursive construction of $\sigma(x)$: those times necessarily belong to 
$\miniop{K(x)-1}{\cup}{i=1}[u_k(x),v_k(x)] $. Thus, we choose $\theta,C$ as in  Theorem~\ref{lemme-pointssourcescontact} and get
\begin{eqnarray*}
&& \Pbarre_{\lambda}(\sigma(x)\ge C\|x\|+t) \\
& \le &\Pbarre_{\lambda} \left( \{s\le C\|x\|+t:\; (0,0)\to (x,s)\to\infty\}\subset\miniop{K(x)-1}{\cup}{i=1}[u_k(x),v_k(x)]\right)\\
& \le & \Pbarre_{\lambda}(\Leb(\{s\le C\|x\|+t:\; (0,0)\to (x,s)\to\infty\})\le\theta t)\\
&& +\Pbarre_{\lambda} \left(\miniop{K(x)-1}{\sum}{i=1}(v_k(x)-u_k(x))>\theta t \right).
\end{eqnarray*}
Lemma~\ref{lemme-pointssourcescontact} allows to control the first term. To control the second one with a Markov inequality, it is sufficient to prove the existence of exponential moments for $\displaystyle \miniop{K(x)-1}{\sum}{i=1}(v_k(x)-u_k(x))$. To do so, we apply the abstract restart Lemma~\ref{restartabstrait}. We define,  for each subset $B$ in $\Zd$, $F^B=0$ and
\begin{eqnarray*}
T^B & = & \inf\{t >\tau^x: \; x \in \xi_t^B\}, \\
G^B & = & \tau^x.
\end{eqnarray*}
Estimate~(\ref{uniftau}) ensures that for each $\lambda \in \Lambda$,
$$\P_\lambda(T^B=+\infty)\ge \P_\lambda( \tau^x=+\infty)\ge \rho>0,$$
and estimate~(\ref{retouche}) ensures the existence of $\alpha>0$ and $c<1$ -- that do not depend on $B$ -- such that  for each $\lambda \in \Lambda$,
$$\E_\lambda[\exp(\alpha G^B)\1_{\{T^B<+\infty\}}]\le\E_\lambda[\exp(\alpha \tau^x) \1_{\{\tau^x<+\infty\}}] = \E_{x.\lambda}[\exp(\alpha \tau^0) \1_{\{\tau^0<+\infty\}}]\le c.$$
Then, with the notation of Lemma~\ref{restartabstrait}, we have
\begin{eqnarray*}
 && \E_\lambda \left[  \exp \left( \alpha \miniop{K(x)-1}{\sum}{i=1}(v_k(x)-u_k(x)) \right)  \right] 
 =  \E_\lambda \left[  \exp \left( \alpha \miniop{K(x)-1}{\sum}{i=0} \tau^x \circ T_k  \right)  \right] 
 \le  \frac{1}{1-c}.
\end{eqnarray*}
To conclude, we note, using~(\ref{uniftau}),  that $\Ebarre_\lambda(.) \le \E_\lambda(. )/\rho$.
\subsection{Proof of Theorem~\ref{lemme-pointssourcescontact}}

We will include in the contact process a block event percolation: sites will correspond to large blocks in   $\Z^d \times [0,\infty)$, and the opening of the bonds will depend of the occuring of  good events that we define now.

\subsubsection{Good events}
\begin{figure}[h!]
\caption{The good event $A(\bar{n_0},u,x_0,x_1)$.} 
\end{figure}

Let $C_1>0$ and $M_1>0$ be fixed. \\
Let $I \in \N^*$, $L \in \N^*$ and $\delta>0$ such that $I \le L$ and $\delta<C_1L$. For $\bar{n_0} \in \Zd$, $x_0,x_1\in [-L,L[^d$ and $u \in \Zd$ such that $\|u\|_1\le1$, we define the following event: 
\begin{eqnarray*}
A(\bar{n_0},u,x_0,x_1) & = & A_{I,L,\delta}^{C_1,M_1}(\bar{n_0},u,x_0,x_1) \\
& = & \left\{
	\begin{array}{c}
	\exists t \in [0,C_1L-\delta] \quad 2L\bar{n_0}+x_1 \in \xi_t^{2L\bar{n_0}+x_0+[-I,I]^d} \\
\omega_{2L\bar{n_0}+x_1}([t,t+\delta])=0\\
	\exists s \in 2L(\bar{n_0}+u) +[-L,L]^d \quad s+[-I,I]^d \subset \xi^{2L\bar{n_0}+x_1}_{C_1L-t}\circ \theta_t \\
	\bigcup_{t \in [0,C_1L]} \xi_t^{2L\bar{n_0}+[-L-I,I+L]^d} \subset 2L\bar{n_0}+[-M_1L,M_1L]^d
         \end{array}
\right\}.
\end{eqnarray*}
We let  then $T=C_1L$. When the event $A(\bar{n_0},u,x_0,x_1)$ occurs, we denote by $s(\bar{n_0},u,x_0,x_1)$ a point  $s$ satisfying the last  condition that defines the event. Else, we let $s(\bar{n_0},u,x_0,x_1)=\infty$.

If this event occurs, then:
\begin{itemize}
 \item Starting from an area of size $I$ centered at a starting point $2L\bar{n_0}+x_0$ in the box with spatial coordinate $\bar{n_0}$, the process at time $T$ colonizes an area of size $I$ centered around the exit point $2L(\bar{n_0}+u)+s(\bar{n_0},u,x_0,x_1)$ in the box with spatial coordinate $\bar{n_0}+u$.
\item Moreover, the point $2L\bar{n_0}+x_1$ is occupied between time $0$ and time $T$ in a time interval with duration at least $\delta$.
\item The realization of this event only depends on what happens in the space-time box  $(2L\bar{n_0}+[-M_1L,M_1L])\times[0,T]$.
\end{itemize}

Let us give a summary of the different parameters:

\vspace{2mm}
\noindent
\begin{tabular}{|l|l|}
 \hline
$L$ & spatial scale of the macroscopic boxes  \\
$I$ & size of the entrance area and of the exit area $(I\le L$) \\
$T$ & temporal size of the  macroscopic boxes ($T=C_1L$)\\
$\delta$ & minimum duration for the infection of $x_1$\\
$\bar{n_0}$ & macroscopical spatial coordinate (coordinate of the big box) \\
$u$ & direction of move ($\|u\|_1\le 1$) \\
$x_0$ & relative position of the entrance area in the box ($x_0\in [-L,L[^d$) \\
$x_1$ & relative position of the target point ($x_1\in [-L,L[^d$)\\
$s(\bar{n_0},u,x_0,x_1)$ & relative position of the exit area  in the box \\
& with coordinate $(\bar{n_0}+u)$ ($s(\bar{n_0},u,x_0,x_1) \in [-L,L[^d$) \\
\hline
\end{tabular}

\begin{lemme}
\label{bonev1}
We can find  constants $C_1>0$ and $M_1>0$ such that we have the following property.

For each $\varepsilon>0$, we can choose, in that specific order, two integers $I \le L$ large enough  and $\delta>0$ small enough  such that for every $\lambda\in \Lambda$, $\bar{n_0} \in \Zd$, and each $u \in \Zd$ with $\|u\|_1\le 1$,  
$$\forall x_0,x_1\in [-L,L[^d\quad \P_\lambda(A(\bar{n_0},u,x_0,x_1))\ge 1-\varepsilon.$$
Moreover, as soon as $\|\bar{n_0}-\bar{n_0'}\|_\infty \ge 2M_1+1$, for every $u,u',x_0,x_0',x_1$,
$$\text{ the events } A(\bar{n_0},u,x_0,x_1) \text{ and } A(\bar{n_0}',u',x_0',x_1) \text{ are independent.}$$
\end{lemme}

\begin{proof}
Let us first note that 
$$\P_\lambda(A(\bar{n_0},u,x_0,x_1))=\P_{2L\bar{n_0}.\lambda}(A(0,u,x_0,x_1)),$$
which permits to assume that $\bar{n_0}=0$. 
Let $\epsilon>0$ be fixed. We first choose $I$ large enough to have 
\begin{equation}
\label{choixI} 
\forall x \in \Zd \quad \P_{\lambda_{\min}}(\tau^{x+[-I,I]^d}=+\infty) \ge 1-\epsilon/4.
\end{equation}
We let  $\epsilon'=\epsilon/(2I+1)^d$. \\
By the FKG Inequality,
$\P_{\lambda_{\min}}(\forall y\in [-I,I]^d, \tau^y=+\infty)>0$.
Translation invariance gives  then $$\miniop{}{\lim}{L\to +\infty}
\P_{\lambda_{\min}}(\exists n\in [0,L]: \; \forall y\in ne_1+[-I,I]^d, \tau^y=+\infty)=1.$$
Let then $L_1$ be such that
$$\P_{\lambda_{\min}}(\exists n\in [0,L]; \forall y\in ne_1+[-I,I]^d, \tau^y=+\infty)>1-\frac{\epsilon'}{12}\P_{\lambda_{\min}}(\tau^0=+\infty).$$
By a time-reversal argument, we have for each $t>0$,
\begin{eqnarray*}
& & \P_{\lambda_{\min}}(\exists n\in [0,L]:\;  ne_1+[-I,I]^d\subset \xi^{\Zd}_t)\\
& = & \P_{\lambda_{\min}}(\exists n\in [0,L]:\; \forall y\in ne_1+[-I,I]^d, \tau^y\ge t)> 1-\frac{\epsilon'}{12}\P_{\lambda_{\min}}(\tau^0=+\infty).
\end{eqnarray*}
We have for each $t\ge 0$ and each $\lambda\in\Lambda$:
\begin{eqnarray*}
& &\P_{x_1.\lambda}(\tau^{0}=+\infty,\; \forall n\in [0,L], \, 2Lu-x_1+ne_1+[-I,I]\not\subset\xi^{0}_t)\\ 
& \le & \P_{x_1.\lambda}(\forall n\in [0,L],\, 2Lu-x_1+ne_1+[-I,I]\not\subset\xi^{\Zd}_t)\\
& & \quad \quad + \P_{x_1.\lambda}(\tau^{0}=+\infty,\,[-(I+4L),(I+4L)]^d\not\subset K'_t)\\
& \le & \P_{\lambda_{\min}}(\forall n\in [0,L], \, ne_1+[-I,I]\not\subset\xi^{\Zd}_t)\\
& & \quad \quad + \P_{x_1.\lambda}(\tau^0=+\infty,\,[-(4L+I),(4L+I)]^d\not\subset K'_t).
\end{eqnarray*}
Let $C>0$ be large enough to satisfy properties~\eqref{asigma} and~\eqref{momtprimeeq}.
Then, with~\eqref{momtprimeeq}, we can find $L_2\ge L_1$ such that
for $L\ge L_2$ and $t\ge 5CL$, we have
$$\Pbarre_{x_1.\lambda}(\exists n\in [0,L]; 2Lu-x_1+ne_1+[-I,I]\subset\xi^{0}_t)\ge 1-\epsilon'/6.$$
Let $\delta>0$ such that $1-e^{-\delta}\le \P_{\lambda_{\min}}(\tau^0=+\infty)\epsilon'/6$ and $\delta<5CL$: if we let 
$$F_t=\big\{\omega_{0}([0,\delta])=0;\; \exists n\in [0,L], \, 2Lu-x_1+ne_1+[-I,I]\subset\xi^{0}_t\big\},$$
we also have, for each $\lambda\in\Lambda$ and each $t\ge 5CL$, that $\Pbarre_{x_1.\lambda}(F_t)\ge 1-\epsilon'/3$.

Then, with Proposition~\ref{magic}, one deduces that if $y\in x+[-I,I]^d$, then
$$\Pbarre_{y.\lambda}(\sigma(x_1-y)\le 4CL, \; \tilde{\theta}_{x_1-y}^{-1} (F_{9CL-\sigma(x_1-y)})) \ge \Pbarre_{y.\lambda}(\sigma(x_1-y)\le 4CL)(1-\epsilon'/3).$$

Considering  estimate~(\ref{asigma}), we can choose $L_3\ge L_2$ such that for $L\ge L_3$, we have
$$\Pbarre_{y.\lambda}(\sigma(x_1-y)\le 4CL, \;\tilde{\theta}_{x_1-y}^{-1} (F_{9CL-\sigma(x_1-y)}))\ge 1-\epsilon'/2.$$
Let $C_1=9C$. With~\eqref{choixI} and the definition of $\epsilon'$, we get
$$
\P_{\lambda}\left(
	\begin{array}{c}
	\exists t \in [0,C_1L-\delta]: \quad x_1 \in \xi_t^{x_0+[-I,I]^d} \\
\omega_{x_1}([t,t+\delta])=0\\
	\exists s \in 2Lu +[-L,L]^d \quad s+[-I,I]^d \subset \xi^{x_1}_{C_1L-t}\circ \theta_t
\end{array}
\right)\ge 1-3\epsilon/4.$$

Finally, one takes for $M$ the  constant given by equation~(\ref{richard}) and lets $M_1=MC_1+2$. With~(\ref{richard}), we can find  $L \ge L_3$ sufficiently large to have  for each $\lambda \in \Lambda$:  
\begin{equation}
\label{choixM} 
\P_{\lambda_{\text{max}}}\left( \bigcup_{0 \le t \le C_1L} \xi^{[-L-I,L+I]^d}_t \subset [-M_1L,M_1L]^d \right)\ge 1-\epsilon/4;
\end{equation}
this fixes the integer $L$.

The local dependence of the events comes from the third condition in their definition. This concludes the proof of the lemma.
\end{proof}

\subsubsection{Dependent macroscopic percolation }
We fix $C_1,M_1$ given by Lemma~\ref{bonev1}. We choose $I \in \N^*$, $L \in \N^*$ and $\delta>0$ such that $I \le L$ and $\delta<C_1L$ and we let $T=C_1L$.

Let $x$ in $\Zd$ be fixed. 
We write $x=2L[x]+\{x\}$, with $\{x\}\in [-L,L[^d$ and $[x]\in\Zd$.
We will first, from the events defined in the preceding  subsection, build a field $(^xW^n_{(\bar{k},u)})_{n \ge 0, \bar{k} \in \Zd, \|u\|_1\le1}$.

The idea is to construct a macroscopic  oriented percolation on the bonds of $\Edo\times\N^*$, looking for the realizations, floor by floor, of translates of good events of type $A(.)$. We start from an area centered at $0$ in the box with coordinate $\bar{0}$; for each $u$ such that $\|u\|_1\le 1$, say that the bond between $(\bar{0},0)$ and $(u, 1)$ is open if $A(\bar{0},u,0,\{x\})$ holds; in that case  we obtain an  infected square centered at the exit point $s(\bar{0},u,0,\{x\})$; all bonds in this floor that are issued from another point than $\bar{0}$ are open, with fictive exit points equal to $\infty$. Then we move to the upper floor: for a box $(\bar{y},1)$, look if it contains exit points of bonds that were open at the preceding step. If it is the case, we choose one of these, denoted by $d^x_1(\bar{y})$, open the bond between $(\bar{y},1)$ and $(\bar{y}+u, 2)$ if $A(\bar{y},u,d^x_1(\bar{y}),\{x\})\circ \theta_{T}$ happens and close it otherwise; in the other case we open all bonds issued from that box, and so on for every floor.

Precisely, we let $d^x_0(\bar{0})=0$ and also $d^x_0(\bar{y})=+\infty$ for every $\bar{y} \in\Zd$ that differs from $0$.
Then, for each $\bar{y}\in\Zd$, each $u \in \Zd$ such that $\|u\|_1\le1$ and for each  $n\ge 0$, we recursively define:
\begin{itemize}
\item If $d^x_n(\bar{y})=+\infty$, $^xW^n_{(\bar{y},u)}=1$.
\item Otherwise, 
\begin{eqnarray*}
^xW^n_{(\bar{y},u)} & = & \1_{A(\bar{y},u,d^x_n(\bar{y}),\{x\})} \circ \theta_{nT}, \\
d^x_{n+1}(\bar{y}) & = & \min\{ s(\bar{y}+u,-u,d^x_n(\bar{y}+u),\{x\})\circ \theta_{nT}: \; \|u\|_1\le 1, \; d^x_n(\bar{y}+u)\ne+\infty \}.
\end{eqnarray*}
\end{itemize}
Recall that the definition of the function $s$ has been given with the one of a good event in the preceding subsection.
Then, $d^x_{n+1}(\bar{y})$ represents the relative position of the entrance area for the $^xW^{n+1}_{(\bar{y},u)}$'s, with $\|u\|_1\le1$. We may have several candidates, that are the relative positions of the exit areas of the $^xW^n_{(\bar{y}+u,-u)}$'s;  the $\min$ only plays the role of a choice function.

We thus obtain an oriented percolation process. Among open bonds, only those  corresponding to the realization of  good events are relevant for  the underlying contact process. Let us note however that the percolation cluster starting at  $\bar{0}$ only contains bonds that correspond to the propagation of the contact process.
  
\begin{lemme}
\label{domistoc}
Again, we work with $C_1,M_1$ given by Lemma~\ref{bonev1}.
For each $q<1$, we can choose parameters $I,L, \delta$ such that for each $\lambda \in \Lambda$, and each $x \in \Zd$,
$$\text{the law of }(^xW_e^n)_{(e,n) \in\Edo\times\N^*}\text{under }\P_{\lambda}\text{ is in }\mathcal{C}(2M_1+1,q).$$
\end{lemme}

\begin{proof} 
For each  $n \in \N$, let $\mathcal G_n=\mathcal F_{nT}$, with $T=C_1L$. Let us note that, for each $x,\overline{k} \in \Zd$ and $n \ge 1$, the quantity $d^x_n(\overline{k})$ is $\mathcal{G}_n$-mesurable, and so does ${}^xW^n_{(\overline{k},u)}$. 

Lemma~\ref{bonev1} ensures that the events $A(\bar{k},u,x_0,\{x\})$ and $A(\bar{l},v,x_0',\{x\})$ are  independent as soon as $\|\bar{k}-\bar{l}\|_1 \ge 2M_1+1$; so we deduce that, conditionally to $\mathcal G_n$, the random variables  ${}^xW^{n+1}_{(\overline{k},u)}$ and ${}^xW^{n+1}_{(\overline{l},v)}$ are independent  as soon as $\|\overline{k}-\overline{l}\|_1 \ge 2M_1+1$.

Let now $x,\overline{k} \in \Zd$, $n \ge 0$ and $u \in \Zd$ such that $\|u\|_1\le 1$:
\begin{eqnarray*}
&& \E_{\lambda}[{}^xW^{n+1}_{(\overline{k},u)}|\mathcal{G}_n\vee \sigma({}^xW^{n+1}_ {(\overline{l},v)}, \;  \|v\|_1\le 1, \; \|\overline{l}-\overline{k}\|_1 \ge 2M_1+1)] \\
& = & \E_{\lambda}[{}^xW^{n+1}_{(\overline{k},u)}|\mathcal{G}_n] \\
& = & \1_{\{d^x_n(\overline{k})=+\infty\}}+ \1_{\{d^x_n(\overline{k})<+\infty\}}\P_{\lambda}[{}^xW^{n+1}_{(\overline{k},u)}=1|d^x_n(\overline{k})<+\infty] \\
& = & \1_{\{d^x_n(\overline{k})=+\infty\}}+ \1_{\{d^x_n(\overline{k})<+\infty\}}\P_{\lambda}[A(\overline{k},u,d^x_n(\overline{k}), \{x\})].
\end{eqnarray*}
With Lemma~\ref{bonev1}, we can choose integers $I <L$ and $\delta>0$ in such a way that
$$\E_{\lambda}[{}^xW^{n+1}_{(\overline{k},u)}|\mathcal{G}_n\vee \sigma({}^xW^{n+1}_ {(\overline{l},v)}, \;  \|v\|_1\le 1, \; \|\overline{l}-\overline{k}\|_1 \ge 2M_1+1)] \ge q.$$
This concludes the proof of the lemma.
\end{proof}

For this percolation process, we  denote by $\overline{\tau}^{\bar{k}}$ and $\overline{\gamma}(\theta,\bar{k},\bar{l})$
the lifetime starting from $\bar{k}$ and the essential hitting times of $\bar{l}$ starting from  $\bar{k}$ in the dependent oriented percolation induced by the Bernoulli random field $(^xW_e^n)_{(e,n) \in\Edo\times\N^*}$.

\begin{lemme}
\label{controledep}
We can choose the parameters $I,L,\delta$ such that the following holds:
\begin{itemize}
\item $\forall \lambda\in\Lambda\quad\P_{\lambda}(\overline{\tau}^0=+\infty)\ge \frac12$.
\item $\forall \lambda\in\Lambda\quad \phi(\lambda)=\E_{\lambda}[e^{\alpha \overline{\tau}^0}\1_{\{\overline{\tau}^0<+\infty\}}] \le 1/2$
\item there exist strictly positive constants $\alpha_0>0,\overline{C}$ such that for every $x,y\in\Zd$
$$\forall \alpha\in [0,\alpha_0]\quad\forall \lambda\in\Lambda\quad \ell(\lambda,\alpha,x,y)=\E_{\lambda}  [\1_{\{\overline{\tau}^{x}=+\infty\}} e^{\alpha \overline{\gamma}(\theta,x,y))}]\le 2 e^{\overline{C}\alpha \|x-y\|}.$$
\end{itemize}

\end{lemme}

\begin{proof}
By Lemma~\ref{petitmomentexpo}, we know that there exist $q<1$ and $\alpha>0$ such that we have
$$\E [e^{\alpha \overline{\tau}^{\bar{0}}}\1_{\{\overline{\tau}^{\bar{0}}<+\infty\}}]\le 1/2$$
for each field in $\mathcal{C}(2M_1+1,q)$.
By  Lemma~\ref{domistoc}, we can choose $I,L,\delta$ such
that $(^xW_e^n)_{(e,n) \in\Edo\times\N^*} \in \mathcal{C}(2M_1+1,q),$
which gives the two first points.
Then, from Lemma~\ref{lineairegamma}, we get constants $A,B,C$
such that for every $x,y\in\Zd$, every $n\ge 0$ and each $\lambda\in\Lambda$, we have
$$\P_{\lambda}(+\infty>\overline{\gamma}(\theta,x,y)> C\|x-y\|_1+n)\le Ae^{-Bn}.$$
We can then find $B'>0$ independent from $x$ and $\lambda$ such that the Exponential law with parameter $B'$ stochastically dominates $(\overline{\gamma}(\theta,x,y)-C\|x-y\|_1)\1_{\{\overline{\gamma}(\theta,x,y)<+\infty\}}$.
Let then $\alpha\le B'/2$: we have
\begin{eqnarray*}
\ell(\lambda,\alpha,x,y) &  =  &e^{\alpha C\|x-y\|_1}\E_{\lambda}[\1_{\{\overline{\tau}^x=+\infty\}}e^{\alpha ((\overline{\gamma}(\theta,x,y)-C\|y-x\|_1))}]\\
& \le & e^{\alpha C\|x-y\|_1} \frac{B'}{B'-\alpha}\\
& \le & 2 e^{\alpha C\|x-y\|_1}.
\end{eqnarray*}
\end{proof}

\subsubsection{Proof of Theorem~\ref{lemme-pointssourcescontact}}

We first choose  $I,L, \delta$ in order to satisfy the inequalities of Lemma~\ref{controledep}, and we let $T=C_1L$.

We use a restart argument. The idea is as follows: fix $\lambda \in \Lambda$ and $x \in \Zd$; if the lifetime  $\tau^0$ of the contact process in random environment is infinite, then one can find by the restart procedure an instant $T_K$ such that 
\begin{itemize}
\item $\xi^0_{T_K}$ contains an area $z+[-2L,2L]^d$, which allows to activate a block oriented percolation, as defined in the previous subsection, from some  $\bar{z}_0\in \Zd$ such that $2\bar{z}_0L+[-L,L]^d \subset z+[-2L,2L]^d$,
\item the block oriented percolation  issued from $\bar{z}_0$ infinitely survives.
\end{itemize}
Then, with Lemma~\ref{lineairegamma}, we give a lower bound for the proportion of time when $\bar{x}_0=[x]$ is occupied by descendents having themselves infinite progeny. By the definition of good events, this will allow to bound from below  the measure of $\{t\ge 0; (0,0)\to (x,t)\to \infty\}$ in the contact process. Indeed, recall that the definition of the event $A(\bar{x}_0,u,x_0,\{x\})$ targets $\{x\}$ and ensures that each time the site $\bar{x}_0=[x]$ is occupied in the macroscopic oriented percolation, then the contact process occupies the site $2L\bar{x}_0+\{x\}=x$ during $\delta$ units of time.

\medskip
\noindent
\underline{Definition of the restart procedure.}
We define the following stopping times: for each non-empty subset $A\subset\Zd$, 
\begin{eqnarray*}
U^A & = & \left\{ 
\begin{array}{l}
T \text{ if } \forall z \in \Zd \; z+[-2L,2L]^d \not\subset \xi^A_T, \\
T(1+\bar{\tau}^0 \circ T_{2\bar{x}^AL} \circ \theta_{T}) \text{ otherwise}\\
\quad \quad \quad \text{with } \bar{x}^A=\inf\{\bar{m} \in \Zd: \; 2\bar{m}L+[-L,L]^d\subset \xi^A_T\}
\end{array}
\right.  \\
\text{and }U^\varnothing & = & +\infty.
\end{eqnarray*}
In other words, starting from a set $A$, we ask if the contact process contains an area in the form  $2\bar{m}L+[-L,L]^d$ at  time $T$, : if the answer is no, we stop, otherwise we consider the lifetime of the macroscopic  percolation issued from the macroscopic site corresponding to that area. Particularly,  if $A \neq \varnothing$ and $U^A=+\infty$, then there exists, at time $T$, in the contact process issued from $A$, an area  $2\bar{x}^AL+[-L,L]^d$ which is fully occupied, and such that the macroscopic oriented percolation issued from thae  macroscopic site $\bar{x}^A$ percolates. We then search in that infinite cluster not too large a time when the proportion of individuals living at $\bar{x}_0=[x]$ and having infinite progeny becomes sufficiently large: 
if $A \neq \varnothing$ and $U^A=+\infty$, we note 
$$R^A=R^A(x)=
\left\{
\begin{array}{ll}
T(1+\overline{\gamma}(\theta,\bar{x}^A,\bar{x}_0)) & \text{ if }A \neq \varnothing \text{ and }U^A=+\infty; \\
0 & \text{ otherwise }.
\end{array}
\right.
$$
Thus, when $U^A=+\infty$, the variable  $R^A$ represents the first time (in the scale of the contact process, not that of the macroscopic oriented percolation) when  the proportion of individuals living at $\bar{x}_0=[x]$ and having infinite progeny becomes sufficiently large.

\medskip
\noindent
\underline{Estimates for the restart procedure.}

\begin{lemme}
There exist constants $\alpha>0$, $q>0$, $c<1$, $A',h>0$ such that for each $\lambda \in \Lambda$, each $A\subset \Zd$, and each $x \in \Zd$,
\begin{eqnarray}
\P_{\lambda}(U^A=+\infty) & \ge & q; \label{restart-q} \\
\E_{\lambda}[\exp(\alpha U^A)\1_{\{U^A<+\infty\}}] & \le & c; \label{restart-c} \\
\E_{\lambda}[\exp(\alpha R^A(x))\1_{\{U^A=+\infty\}}]  & \le &  A'e^{\alpha h(\|\bar{x}_0\|_\infty+ \|A\|_\infty)}. \label{restart-A}
\end{eqnarray}
\end{lemme}

\begin{proof}
We easily get  (\ref{restart-q}) from a stochastic comparison:  for each $\lambda \in \Lambda$ and each  non-empty $A$,
$$
\P_{\lambda}(U^A=+\infty)  \ge   \P_{\lambda_{\text{min}}}([-2L,2L]^d \subset\xi^0_T)\P(\bar{\tau}^0=+\infty)=q>0.
$$
Now, if $\alpha>0$ , $A \subset \Zd$ is non-empty and $\lambda \in \Lambda$, we have with Lemma~\ref{controledep},
\begin{eqnarray*}
&& \E_{\lambda}[\exp(\alpha U^A)\1_{\{U^A<+\infty\}}]   \\
& = & e^{\alpha T}\left( 1-\P_{\lambda}(\exists z \in \Zd, \; z+[-2L,2L]^d \subset \xi^A_T)\left(1- \E[e^{\alpha T \bar{\tau}^0} \1_{\{\bar{\tau}^0<+\infty\}}]\right)\right) \\
& \le & e^{\alpha T}\left( 1-\frac12\P_{\lambda}(\exists z \in \Zd, \; z+[-2L,2L]^d \subset \xi^A_T) \right) \\
& \le & e^{\alpha T}\left( 1-\frac12\P_{\lambda_{\text{min}}}(\exists z \in \Zd, \; z+[-2L,2L]^d \subset \xi^A_T) \right)=c <1
\end{eqnarray*}
provided that $\alpha>0$ is small enough; this proves (\ref{restart-c}).

By the strong Markov property and Lemma~\ref{controledep}, if we choose $\alpha>0$ small enough, then for each $\lambda \in \Lambda$,
\begin{eqnarray}
&& \E_{\lambda} [ \exp(\alpha R^A) \1_{\{U^A=+\infty\}} | \mathcal{F}_{T} ]  \nonumber \\
& = & \1_{\{\exists z \in \Zd, \; z+[-2L,2L]^d \subset \xi^A_T\}} e^{\alpha T} \E[ \exp(\alpha T \overline{\gamma}(\theta,\bar{x}^A,\bar{x}_0))\1_{\{\bar{\tau}^{\bar{x}^A}=+\infty\}}]\nonumber\\
& \le & 2 e^{\alpha T} \exp(\overline{C} \alpha T  \|\bar{x}^A-\bar{x}_0\|_\infty) \nonumber \\
 &\le &  2 e^{\alpha T(1+\overline{C}\|\bar{x}_0\|_\infty)}\exp(\overline{C} \alpha T  \|\xi^A_T\|_\infty). \label{laun}
\end{eqnarray}
We use the comparison with Richardson's model to bound the mean of the last term: let us choose the positive constants $M,\beta$ such that
$$\forall s,t\ge 0\quad \P_{\lambda_{\max}}(\|\xi^0_s\|_{\infty}\ge Ms+t)\le e^{-\beta t}.$$
Then,  for each  non-empty finite set $A$, each $t>0$, and each $\lambda \in \Lambda$,
\begin{eqnarray*}
\P_{\lambda}(\|\xi^A_T\|_{\infty}\ge 2\|A\|_\infty + MT+t) & \le & \P_{\lambda_{\max}}(\max_{a \in A} \|\xi^a_T -a \|_{\infty}\ge \|A\|_\infty+MT+t) \\
& \le & \Card{A} \P_{\lambda_{\max}}(\|\xi^0_T\|_{\infty}\ge MT+\|A\|_\infty+t)\\
& \le & \|A\|_\infty^d e^{-\beta (\|A\|_\infty+t)}\le \alpha' \exp(-\beta t). 
\end{eqnarray*}
Then, for  $\alpha$ small enough,
\begin{eqnarray}
\E_\lambda[\exp(\overline{C}\alpha T \|\xi^A_T\|_\infty)
] 
& \le & e^{\overline{C}\alpha T(2\|A\|_\infty + MT)} \left(1+\frac{\overline{C}\alpha T\alpha'}{\beta-\overline{C}\alpha T} \right) \nonumber \\
& \le & 2e^{\overline{C}\alpha T(2\|A\|_\infty + MT)}. \label{ladeux}
\end{eqnarray}
Inequality~(\ref{restart-A}) immediately follows from (\ref{laun}) and (\ref{ladeux}).
\end{proof}

\medskip
\noindent
\underline{Application of the restart lemma~\ref{restartabstrait}.}
Let  
\begin{eqnarray*}
T_0=0 \text{ and } T_{k+1} & = & 
	\begin{cases}
	+\infty & \text{if }T_k=+\infty\\
	T_k+U^{\xi_0^{T_k}} \circ \theta_{T_k} & \text{otherwise;}
	\end{cases} \\
K & = & \inf\{k\ge 0:\;T_{k+1}=+\infty\}.
\end{eqnarray*}
The restart lemma, applied with $T^.=G^.=U^.$ and $F^.=0$,  ensures that 
\begin{eqnarray*}
\E_\lambda[\exp(\alpha T_K)] & \le &  \frac{A'}{1-c}.
\end{eqnarray*}
Applying now the restart lemma with $G^.=0$ and $F^.=R^.$, we get  that 
\begin{eqnarray*}
 \E_\lambda[\exp(\alpha (R^{\xi_0^{T_K}} \circ\theta_{T_{K}}-(h\|\bar{x}_0\|_\infty+\|\xi_0^{T_K}\|_\infty )))] 
& \le & \frac{A'}{1-c} .
\end{eqnarray*}
Particularly, it holds that for each $s>0$ and $t>0$,
\begin{eqnarray}
\P_\lambda(T_K>s) & \le & \frac{A'}{1-c} \exp(-\alpha s); \label{queueTK}\\
\P_\lambda
\left(
\begin{array}{c}
R^{\xi_0^{T_K} \circ\theta_{T_{K}}}\ge t/2, \\ 
T_K \le s, \; H^0_s \subset B^0_{Ms}
\end{array}
\right) & \le & 
\frac{A'}{1-c} \exp(\alpha (h (\|\bar{x}_0\|_\infty+Ms )- t/2))  \label{queueR}.
\end{eqnarray}
On the event $\{\tau^0=+\infty\}$, one can be sure that the contact process is non-empty at each step of the restart procedure : the restart Lemma ensures that  at time  $T_K+T$, one can find some area from which the directed block percolation  percolates, and, by construction, that for every $t \ge T_K+R^{\xi_0^{T_K} \circ\theta_{T_{K}}}$,
$$\Leb(\{s \in [T_K+T,t]: \; (0,0) \to(x,s)\to \infty \})\ge \delta \theta \text{Int}(\frac{t-(T_K+T)}{T})\ge  \frac{\delta \theta}{2T} t$$
as soon as $T_K\le t/2-1$.

Let  $C=\frac{2h}L$. Let now be $x \in \Zd$, and $t \ge C\|x\|_\infty$.
\begin{eqnarray*}
& & \P_\lambda \left(\tau^0=+\infty, \Leb(\{s \in [0,t]: \; (0,0) \to(x,s)\to \infty \})<\frac{\delta \theta}{2T} t \right) \\
& \le & \P_\lambda(T_K> t/2-1)+\P_\lambda(T_K\le  t/2-1, \; t<T_K+R^{\xi_0^{T_K}} \circ\theta_{T_{K}})\\
& \le & \P_\lambda(T_K> t/2-1)+\P_\lambda(R^{\xi_0^{T_K}} \circ\theta_{T_{K}}>t/2).
\end{eqnarray*}
We control the first term with (\ref{queueTK}). For the second one, we take $s=\frac{t}{8hM}$:
\begin{eqnarray*}
 &&\P_\lambda(R^{\xi_0^{T_K}} \circ\theta_{T_{K}}>t/2) \\
& \le & \P_\lambda(R^{\xi_0^{T_K} \circ\theta_{T_{K}}}> t/2, \; T_K \le s, \; H^0_s \subset B^0_{Ms}) + \P_\lambda(T_K > s)+\P_\lambda(H^0_s \not\subset B^0_{Ms}).
\end{eqnarray*}
We control the last two terms  with (\ref{queueTK}) and (\ref{richard}); for the first one, we use (\ref{queueR}): since 
 $\|\bar{x}_0\|_\infty \le \frac1{2L} \|x\|_\infty+1$, 
 \begin{eqnarray*}
 \P_\lambda \left(
\begin{array}{c}
R^{\xi_0^{T_K} \circ\theta_{T_{K}}}> t/2, \\
 T_K \le s, \; H^0_s \subset B^0_{Ms}
\end{array}
\right) & \le & 
 \frac{A'}{1-c} \exp(\alpha (h (\|\bar{x}_0\|_\infty+Ms )- t/2)) \\
 & \le & \frac{A'e^{\alpha h}}{1-c} \exp\left( \alpha \left(   \left( \frac{h}{2L}\|x\|_\infty-\frac{t}4\right) -\frac{t}{8} \right) \right)\\
  & \le & \frac{A'e^{\alpha h}}{1-c}\exp(-\alpha t/8),
 \end{eqnarray*}
which concludes the proof.

\section{Lower large deviations}
\subsection{Duality}

We have seen that the hitting times $\sigma(nx)$ enjoy superconvolutive properties. In a deterministic frame, Hammersley~\cite{MR0370721} has proved 
that the superconvolutive property allows to express the large deviation functional in terms of the moments generating function, as in Chernoff's Theorem. 
We will see that this property also holds in an ergodic random environment.
The following proof is inspired by Kingman~\cite{MR0438477}.

\begin{proof}[Proof of Theorem~\ref{LDP}]
Since
$\{t(x)\le t,\; \tau^x\circ \theta_{t(x)}=+\infty\}\subset\{\sigma(x)\le t\}\subset\{t(x)\le t\},$ the Markov property ensures that
$$\Pbarre_{\lambda}(t(x)\le t)\P_{\lambda}(\tau^x=+\infty)\le\Pbarre_{\lambda}(\sigma(x)\le t)\le\Pbarre_{\lambda}(t(x)\le t).$$
Thus, letting $R=-\log \P_{\lambda_{\min}}(\tau^0=+\infty)$,
we have
\begin{equation}
\label{equationunun}
 -\log(\Pbarre_{\lambda}(t(x)\le t))\le -\log(\Pbarre_{\lambda}(\sigma(x)\le t))\le -\log(\Pbarre_{\lambda}(t(x)\le t))+R.
\end{equation}
Similarly, 
$$ \E_{\lambda}[e^{-t(x)}] \ge  \E_{\lambda}[e^{-\theta\sigma(x)}]  \ge  \E_{\lambda}[\1_{\{\tau^x\circ\theta_{t(x)}=+\infty\}} e^{-\theta t(x)}]=\E_{\lambda}[ e^{-\theta t(x)}]\P_{\lambda}(\tau^x=+\infty),$$
which leads to 
\begin{equation}
\label{equationquatrequatre}
 -\log \Ebarre_{\lambda}[e^{-t(x)}] \le -\log \Ebarre_{\lambda}[e^{-\theta\sigma(x)}]  \le  -\log \Ebarre_{\lambda}[ e^{-\theta t(x)}]+R.
\end{equation}
Then, having a large deviation principle in mind, working with $\sigma$ or $t$ does not matter. We will work here with  $\sigma$, which gives simpler 
relations. We know that
\begin{equation}\\
\label{sousaddd}
 t((n+p)x)\le \sigma(nx)+\sigma(px)\circ\tilde{\theta}_{nx},
\end{equation}
that $\sigma(nx)$ and $\sigma(px)\circ\tilde{\theta}_{nx}$ 
are independent under $\Pbarre_{\lambda}$, and that the law of
$\sigma(px)\circ\tilde{\theta}_{nx}$  under $\Pbarre_{\lambda}$ is the law of $\sigma(px)$ under $\Pbarre_{nx.\lambda}$ (see Proposition~\ref{magic}). Then
\begin{equation}
\label{equationdeuxdeux}
-\log \Pbarre_{\lambda}(t((n+p)x)\le nu+pv)\le -\log \Pbarre_{\lambda}(\sigma(nx)\le nu)-\log \Pbarre_{nx.\lambda}(\sigma(px)\le pv).
\end{equation}
Let $g_n^{x}(\lambda,u)=-\log \Pbarre_{\lambda}(\sigma(nx)\le nu)+R\text{ and }G_n^{x}(u)=\int_{\Lambda} g^{x}_{n}(\lambda,u)\ d\nu(\lambda)$.
Inequalities~(\ref{equationunun}) and (\ref{equationdeuxdeux}) ensure that
\begin{equation}
\label{equationtroistrois}
 g^{x}_{n+p}(\lambda,u)\le g^{x}_n(\lambda,u)+g^{x}_p( T_{x}^n\lambda,u).
\end{equation}
Since $0\le g^{x}_1(\lambda,u) \le  -\log \Pbarre_{\lambda_{\min}}(\sigma(x)\le u)+R<+\infty$, Kingman's subadditive ergodic theorem ensures that
$\frac{g^{x}_{n}(u,\lambda)}{n}$ converges to 
$$\Psi_x(u)=\inf_{n\ge 1}\frac1{n} G^x_n(u)=\lim_{n\to+\infty} \frac1{n}G^x_n(u).$$
for  $\nu$-almost every $\lambda$.

Note that (\ref{equationtroistrois}) ensures that for every $n,p\in\N$ and every $u,v>0$, 
$$\Psi_x\left(\frac{nu+pv}{n+p}\right)\le \frac1{n+p}G^x_{n+p}\left(\frac{nu+pv}{n+p}\right)\le \frac{n}{n+p} \frac{G^x_n(u)}{n}+\frac{p}{n+p} \frac{G^x_p(v)}{p}.$$
Let $\alpha \in ]0,1[$. Since $\Psi_x$ is non-increasing, considering  some sequence  $n_k,p_k$ such
$\frac{n_k}{n_k+p_k}$ tends to $\alpha$ from above, we get
$$\Psi_x(\alpha u+(1-\alpha)v)\le \alpha\Psi_x(u)+(1-\alpha)\Psi_x(v).$$
So $\Psi$ is convex.

Similarly, let $h_n^{x}(\lambda,\theta)=-\log \Ebarre_{\lambda}[e^{-\theta \sigma(nx)}]+R\text{ and }H^x_n({\theta})=\int h^{x}_{n}(\lambda,\theta)\ d\nu(\lambda)$. As previously, with (\ref{equationquatrequatre}) and the subadditive relation~(\ref{sousaddd}), 
we have
\begin{eqnarray*}\Ebarre_{\lambda}[e^{-\theta\sigma((n+p)x)}]&\ge & e^{-R}\Ebarre_{\lambda}[ e^{-\theta t((n+p)x)}]\\
& \ge & e^{-R}\Ebarre_{\lambda}[ e^{-\theta (\sigma(nx)+\sigma(px)\circ\tilde{\theta}_{nx})}]
 =  e^{-R}\Ebarre_{\lambda}[ e^{-\theta \sigma(nx)}]\Ebarre_{\lambda}[e^{-\theta\sigma(px)}],
\end{eqnarray*}
and then the inequality
$$h^{x}_{n+p}(\lambda,{\theta})\le h^{x}_n(\lambda,{\theta})+h^{x}_p (T_{x}^n\lambda,{\theta}).$$
Since  $0\le h^{x}_1(\lambda,{\theta})\le  -\log \Ebarre_{\lambda_{\min}}[e^{-\theta\sigma(x)}]<+\infty$, Kingman's subadditive ergodic theorem ensures that for $\nu$-almost every $\lambda$,
$\frac{h^{x,{\theta}}_{n}(\lambda,{\theta})}{n}$ converges to
$$K_x({\theta})=\inf_{n\ge 1}\frac1{n} H^x_n({\theta})=\lim_{n\to+\infty} \frac1{n}H^x_n({\theta}).$$
Let now  $\theta\ge 0$ and $u>0$. By the Markov inequality, we observe that
\begin{eqnarray*}
 \Pbarre_{\lambda}(\sigma(nx)\le nu)\le e^{\theta nu}\Ebarre_{\lambda}[e^{-\theta\sigma(nx)}], & \ie & -g_n^{x}(.,u)\le\theta nu-h_n^{x,\theta}(.,u), \nonumber \\
& \ie & G_n^x(u)\ge -\theta nu +H_n(\theta), \nonumber \\
& \ie & \Psi_x(u)\ge-\theta u+K_x(\theta).
\end{eqnarray*}
Thus, we easily get
\begin{eqnarray}
\label{lune}
\forall u>0\quad \Psi_x(u)& \ge& \sup_{\theta\ge 0}(K_x(\theta)-\theta u), \\
 \label{lunea} \forall \theta>0 \quad K_x(\theta) & \le & \inf_{u>0} (\Psi_x(u)+\theta u).
\end{eqnarray}
It remains to prove both reversed inequalities.
Let us first prove
\begin{equation}
\label{lautre}
\forall \theta>0 \quad K_x(\theta)\ge \inf_{u>0} \{\Psi_x(u)+\theta u\}.
\end{equation}
Let $\theta>0$. Define $M=\miniop{}{\inf}{u>0} \{\Psi_x(u)+\theta u\}$ and note that for each $u$ and each integer $n$
$$G_n^x(u)+n\theta u\ge n\Psi_x(u)+n\theta u\ge nM.$$
Fix $\epsilon>0$.
Define $E_{n,\epsilon}=\{\lambda: g_n^{x,u}(\lambda)\ge G^x_n(u)-n\epsilon\}$.
We have
 \begin{eqnarray*} 
H^x_n(\theta)  
& \ge & \int_{E_{n,\epsilon}} h^x_n(\theta)\ d\nu(\lambda) 
  =  \int_{E_{n,\epsilon}} (R-\log \Ebarre_{\lambda}[e^{-\theta \sigma(nx)})]\ d\nu(\lambda)\\
&  =  & \int_{E_{n,\epsilon}} -\log \left[\int_0^{+\infty}n\theta e^{-\theta nu}e^{-R}\Pbarre_{\lambda}(\sigma(nx)<nu)\ du\right]\ d\nu(\lambda).
\end{eqnarray*}
For every $\lambda\in E_{n,\epsilon}$ and $b>0$, one has 
 \begin{eqnarray*}\int_0^{+\infty} n\theta e^{-\theta nu}e^{-R}\Pbarre_{\lambda}(\sigma(nx)<nu)\ du
& \le & e^{-\theta n b}+ \int_0^{b} n\theta e^{-\theta nu}e^{-R}\Pbarre_{\lambda}(\sigma(nx)<nu)\ du\\
& \le & e^{-\theta n b}+\int_0^{b} n\theta e^{-\theta nu}e^{-g_n^{x,u}(\lambda)}\ du\\
& \le & e^{-\theta n b}+\int_0^{b} n\theta e^{-\theta nu}e^{-G^x_n(u)+n\epsilon  }\ du\\
& \le & e^{-\theta n b}+n\theta be^{-n(M-\epsilon)}\\
& \le & (nM+1)e^{-n(M-\epsilon)}\quad\text{with }b=M/\theta.
\end{eqnarray*}
Finally,
$$ \frac{H^x_n(\theta)}{n}  \ge \nu(E_{n,\epsilon})\left(-\frac{\log(1+nM)}n+M-\epsilon\right).$$
Since $\nu(E_{n,\epsilon})$ tends to $1$ when $n$ goes to infinity, one deduces that
$$K_x(\theta)=\lim\frac1{n}H^x_n(\theta)\ge M-\epsilon.$$
Letting $\epsilon$ tend to $0$, we get~(\ref{lautre}).

Let us finally prove
\begin{equation}
 \label{lautreb}
\forall u>0\quad \Psi_x(u) \le  \sup_{\theta\ge 0}(K_x(\theta)-\theta u).
\end{equation}
Let $u>0$. It is sufficient to prove that there exists $\theta_u\ge 0$, with $\Psi_x(u)\le -\theta_u u+K_x(\theta_u)$.
Since $\Psi_x$ is convex and non-increasing, there exists a slope $-\theta_u\le 0$
such that $\Psi_x(v)\ge\Psi_x(u)-\theta_u(v-u)$. Then
$$ K_x(\theta_u)=\inf_{v>0} \{\Psi_x(v)+\theta_u v\} \ge \inf_{v>0} \{\Psi_x(u)-\theta_u(v-u)+\theta_u v\}
 \ge  \Psi_x(u)+\theta_u u,$$
which completes the proof of~(\ref{lautreb}) and of the reciprocity formulas.

The function $-K_x(-\theta)$  corresponds to $\Psi_x$ in the Fenchel-Legendre duality: therefore, it is  convex.
Particularly, the functions $\Psi_x$ and $K_x$ are continuous on  $]0,+\infty[$.
By the definition of $\Psi_x$ and $K_x$, there exists $\Lambda'\subset\Lambda$ with $\nu(\Lambda')=1$ and such that for each
$u\in\Q\cap (0,+\infty)$ and each $\theta\in\Q\cap  [0,+\infty)$, we have
\begin{eqnarray*}
&&  \lim_{n\to +\infty} -\frac1{n}\log \Pbarre_{\lambda}(\sigma(nx)\le nu)  =  \Psi_x(u), \\
\text{and } && \lim_{n\to +\infty} -\frac1{n}\log \Ebarre_{\lambda}e^{-\theta\sigma(nx)} =  K_x(\theta).
\end{eqnarray*}
Since the functions $\theta\mapsto h^{x,\theta}_n$ and $u\mapsto h^{x,u}_n$ are monotonic and their limits $\Psi_x$ and $K_x$ are continuous, it is easy to check that the convergences also hold for every
$\lambda\in\Lambda'$,  $u>0$ and $\theta\ge 0$.
\end{proof}

\subsection{Lower large deviations}

We prove here Theorem~\ref{dessouscchouette}. 
Remember that ${\P}(.)=\int_\Lambda {\P}_\lambda(.)\ d\nu(\lambda)$. 
The main step is actually to prove the following:
\begin{theorem}
\label{GDdessous}
Assume that $\nu=\nu_0^{\otimes \Ed}$ and that the support of $\nu_0$ is included in $[\lambda_{\text{min}}, \lambda_{\text{max}}]$.
For every $\epsilon>0$, there exist $A,B>0$ such that
$$\forall t \ge 0 \quad \P (\xi^0_t \not \subset (1+\epsilon)t A_\mu)\le A\exp(-Bt).$$
\end{theorem}
Using the norm equivalence on $\Rd$, we introduce constants $C^-_\mu,C^+_\mu>0$ such that
\begin{equation}
 \label{normes}
\forall z \in \Rd \quad C^-_\mu\|z\|_\infty \le \mu(z) \le C^+_\mu \|z\|_{\infty}.
\end{equation}

Let $\alpha,L,N,\epsilon>0$. We define the following event, relatively to the space-time box $B_N=B_N(0,0)=[-N,N]^d\times[0,2N]$:
\begin{eqnarray*}
A^{\alpha,L,N,\epsilon}= \left\{\forall (x_0,t_0) \in B_N\quad 
  \xi^{x_0}_{\alpha L N-t_0}\circ \theta_{t_0} \subset x_0+(1+\epsilon)(\alpha L N-t_0)A_{\mu}\right\}\cap\\
 \left\{\forall (x_0,t_0) \in B_N
 \miniop{}{\cup}{0\le s\le\alpha L N-t_0}  \xi^{x_0}_{s}\circ \theta_{t_0} \subset ]-LN,LN[^d\right\}.
\end{eqnarray*}
The first part of the event ensures that the descendants, at time $\alpha L N$, of any point $(x_0,t_0)$ in the box $B_N$ are included in $x_0+(1+\epsilon)(\alpha L N)A_{\mu}$: it is a sharp control, requiring the asymptotic shape theorem. The second part ensures that the descendants, at all times in $[0,\alpha L N] $, of the whole box $B_N$ are included in $]-LN,LN[^d$: the bound is rough, only based on the (at most) linear growth of the process. 

We say that the box $B_N$ is good if  $A^{\alpha,L,N,\epsilon}$ occurs. We also define, for $k \in \Zd$ and $n \in \N$, the event $A^{\alpha,L,N,\epsilon}(k,n)=A^{\alpha,L,N,\epsilon} \circ T_{2kN} \circ \theta_{2nN}$ and we say that the box $B_N(k,n)$ is good if the event $A^{\alpha,L,N,\epsilon}(k,n)$ occurs. 

The proof of the lower large deviation inequalities is close to the one by Grimmett and Kesten~\cite{grimmett-kesten} for first passage-percolation. If a point $(x,t)$ is infected too early, it means that its path of infection has ``too fast'' portions when compared with the speed given by the asymptotic shape theorem. For this path, we build a sequence of boxes associated with path portions, and the existence of a ``too fast portion'' forces the corresponding box to be bad. But we are going to see that we can choose the parameters to ensure that 
\begin{itemize}
\item the probability under $\P$ for a box to be good is as close to $1$ as we want,
\item the events ``$B_N(k,0)$ is good'' are only locally dependent. 
\end{itemize}
We then conclude the proof by a comparison with independent percolation with the help of the Liggett--Schonmann--Stacey Lemma~\cite{LSS} and a control of the number of possible sequences of boxes.  

\begin{lemme}
\label{catendvers12}
We have
\begin{itemize}
\item The events $(\{B_N(k,0) \text{ is good}\})_{k\in \Zd}$ are identically distributed under~$\P$. 
\item There exists $\alpha>0$ such that for every $\epsilon \in (0,1)$, there exists an integer $L$ (that can be taken as large as we want) such that
$$\lim_{N \to +\infty} \P(A^{\alpha,L,N,\epsilon})=1.$$
\item 
If moreover $\nu=\nu_0^{\otimes \Ed}$, then the events  $(\{B_N(k,0) \text{ is good}\})_{k\in \Zd}$ are $(L+1)$-dependent under $\P$.
\end{itemize}
\end{lemme}

\begin{proof}
The first and last points are clear. Let us prove the second point.
The idea is to find a point $(0,-k)$, with $k$ large enough, such that
\begin{itemize}
\item the descendants of $(0,-k)$ are infinitely many and behave correctly (without excessive speed) 
\item the coupled region of $(0,-k)$ contains a set of points that is necessarily crossed by any infection path starting from the box $B_N$.
\end{itemize}
Indeed, this will allow to find, for all the descendants of $B_N$, a unique common ancestor, and thus to control the growth of all the descendants of $B_N$ by simply controlling the descendants of this ancestor. A control on a number of points of the order of the volume of $B_N$ will thus be replaced by a control on a single point. See Figure~\ref{uneautrefigure}.
Let $\epsilon>0$ be fixed. 

We first control the positions of the descendants of the box $B_N$ at time $4N$.
Let $A,B,M$ be the constants given by Proposition~\ref{propuniforme}. We recall that $\omega_x$, for $x\in \Zd$, and $\omega_e$, for $e \in \Ed$ are the Poisson point processes giving respectively the death times for $x$ and the potential infection times through edge $e$. We define, for every integer $N$:
\begin{eqnarray*}
\tilde{A}_1^N & = & \{H^0_{4N}\not \subset [-(4M+1)N,(4M+1)N]^d\}, \\
A_1^N & = & \left\{ \sum_{x \in [-N,N]^d}\int \1_{\{ \tilde{A}_1^N \circ T_x \circ \theta_t \} }\  d\left(\delta_0+\sum_{e \ni x}\omega_e\right)(t) =0 \right\}.
\end{eqnarray*}
Note in particular that
\begin{equation}
\label{inclusion1}
 A_1^N \subset  \left\{\forall (x_0,t_0) \in B_N\quad 
 \xi^{x_0}_{4N-t_0}\circ \theta_{t_0} \subset [-(4M+1)N,(4M+1)N]^d \right\} .
\end{equation}
We have with~\eqref{richard}, 
\begin{eqnarray*}
&& \E \left(\sum_{x \in [-N,N]^d}\sum_{e \ni x}\int_0^{2N} \1_{\{ \tilde{A}_1^N \circ T_x \circ \theta_t \} } \ d(\delta_0+\omega_e)(t) \right) \\
& \le & (2N+1)^d2d(1+2N\lambda_{\max}) \P_{\lambda_{\max}}(\tilde{A}_1^N) \\
& \le & (2N+1)^d2d(1+2N\lambda_{\max})A\exp(-4BN), 
\end{eqnarray*}
and thus, with the Markov inequality,
\begin{equation}
\label{probab1}
\lim_{N\to +\infty}   \P(A_1^N) =1.
\end{equation}
With~(\ref{inclusion1}), we deduce that with a large probability, if $N$ is large enough, the descendants of $B_N$ at time  $4N$ are included in $[-(4M+1)N,(4M+1)N]^d$.

Now, we look for points with a good growth (we will look for the common ancestor of $B_N$ among these candidates):
\begin{eqnarray*}
\tilde{A}_2^t & = & \{\tau^0=+\infty, \; \forall s\ge t\quad K'_s\supset (1-\epsilon)s A_{\mu}\text{ and } \xi^0_s\subset (1+\epsilon/2)sA_{\mu}\}, \\
{A}_2^{t,N} & = & \miniop{N-1}{\cup}{k=0} \tilde{A}_2^{t}\circ \theta_{-k}.
\end{eqnarray*}
The first event says that the point $(0,0)$ lives forever and has a good growth after time $t$ (at most linear growth, and at least linear growth for its coupled zone), while the second event says that there exists a point $(0,-k)$ with a good growth and such that $k \in [0..N-1]$.
Theorem 3 in Garet-Marchand~\cite{GM-contact} ensures that $\miniop{}{\lim}{t\to +\infty}\Pbarre(\tilde{A}_2^t)=1$. But
\begin{eqnarray*}
\P(\tilde{A}_2^t) & = & \int \P_{\lambda}(\tilde{A}_2^t) d\nu(\lambda)
  =  \int \Pbarre_{\lambda}(\tilde{A}_2^t)\P_{\lambda}(\tau^0=+\infty) d\nu(\lambda)\\
 & \ge & \int \Pbarre_{\lambda}(\tilde{A}_2^t)\P_{\lambda_{\min}}(\tau^0=+\infty) d\nu(\lambda)
  \ge  \P_{\lambda_{\min}}(\tau^0=+\infty)\Pbarre(\tilde{A}_2^t).
\end{eqnarray*}
So there exists $t_2$ such that $\P(\tilde{A}_2^{t_2})>0$.
As the time translation $\theta_{-1}$ is ergodic under $\P$, we get
\begin{equation}
\label{probab2}
\lim_{N \to +\infty} \P\left( {A}_2^{t_2,N}\right)=\lim_{n\to +\infty}\P\left( \miniop{n-1}{\cup}{k=0} \tilde{A}_2^{t_2}\circ \theta_{-k}\right)=1.
\end{equation}
In other words, with a large probability, if $N$ is large enough there exists $k \in [0..N-1]$ such that the point $(0,-k)$ has a good growth.

\begin{figure}[h!]
\begin{center}
\includegraphics[scale=0.5]{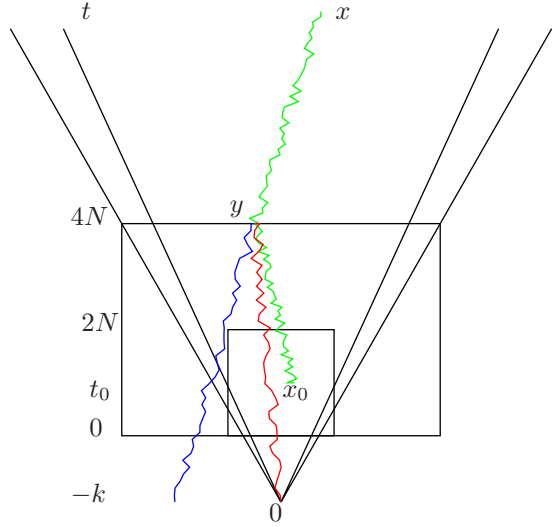}
\put(-230,10){$-k$}%
\put(-223,35){$0$}%
\put(-226,75){$2N$}%
\put(-230,115){$4N$}%
\put(-170,120){$y$}%
\put(-150,50){$x_0$}%
\put(-223,50){$t_0$}
\put(-155,3){$0$}%
\put(-130,193){$x$}%
\put(-226,193){$t$}%
\caption{Coupling from the past}
\label{uneautrefigure} 
\end{center}
\end{figure}
\noindent

Take $L_1=L_1(\epsilon)>0$ such that
\begin{equation}
\label{inclusion}\forall N\ge 1\quad (L_1+1)N(1-\epsilon)A_{\mu}\supset [-(4M+1)N,(4M+1)N]^d.
\end{equation}
Thus, if we find an integer $k\ge \max(t_2,L_1 N)$ such that
$A_{t_2}\circ \theta_{-k}$ occurs, then the descendants of the box $B_N$ at time  $4N$ are in the coupled region of $(0,-k)$.

Denote by $\overleftarrow{\tau}^y$ the life time of $(y,0)$ for the contact process when we reverse time. As the contact process is self-dual, $\overleftarrow{\tau}^y$ as the same law as $\tau^y$. Set
$$A_3^N=\left\{\forall y\in [-(4M+1)N,(4M+1)N]^d\quad  \overleftarrow{\tau}^y\circ{\theta}_{4N}=+\infty \text{ or } \overleftarrow{\tau}^y\circ{\theta}_{4N}<2N\right\}.$$
The control~(\ref{grosamasfinis}) of large lifetimes ensures that 
\begin{equation}
\label{probab3}
\lim_{N \to +\infty} \P(A_3^N)=1.
\end{equation}

Assume now that $N\ge t_2/L_1$. Thus $L_1N\ge t_2$.
Let us see that on $\displaystyle A_1^{N}\cap (A_2^{t_2,N}\circ \theta_{-L_1N})
 \cap A_3^N$, we have
\begin{equation}
\label{attrape}
\forall t\ge 4N \quad  \miniop{}{\cup}{(x_0,t_0)\in B_N} \xi_{t-t_0}^{x_0}\circ\theta_{t_0} \subset ((L_1+1)N+t)(1+\epsilon/2)A_{\mu}.
\end{equation}
Indeed, let $t\ge 4N$ and $x\in \Zd$ be such that  $(x,t)$ is a descendant of  $(x_0,t_0)\in B_N$. Let $(y,4N)$  be an ancestor of $(x,t)$ and a  descendant of $(x_0,t_0)$. On the event $A_1^{N}$, the point $y$ is in $[-4MN,4MN]^d$.
But, on $A_3^N$, the definition of $y$ ensures that $\overleftarrow{\tau}^y\circ{\theta}_{4N}=+\infty$: 
so $(y,4N)$ has a living ancestor a time $-k$, for each $k$ such that $L_1 N \le k \le (L_1+1) N-1$. 
But, on $A_2^{t_2,N}\circ \theta_{-L_1N}$, inclusion~(\ref{inclusion}) ensures that $(y,4N)$ is in the coupled region of $(0,-k)$ for one of these $k$, and so $(y,4N)$ is a descendant of this $(0,-k)$. 
Finally, $(x,t)$ is also a descendant of $(0,-k)$, and, always on  $A_2^{t_2,N}\circ \theta_{-L_1N}$,
$$\mu(x)\le (k+t)(1+\epsilon/2)\le ((L_1+1)N-1+t)(1+\epsilon/2),$$
which proves~(\ref{attrape}).

We then choose $\alpha \in (0,1)$ and an integer  $L$ such that
\begin{eqnarray*}
 \alpha & < & \frac{2C_\mu^-}{3} \le \frac{C_\mu^-}{1+\epsilon/2}, \label{alpha} \\
L & \ge & \max\left\{ \frac4{\alpha}, \; \frac{L_1+1}{ C_\mu^--\alpha(1+\epsilon/2)}, \; 4M+1, \; \frac{2}{\alpha \epsilon}((L_1+1)(1+\epsilon/2)+C_\mu^++2 \right\}. \label{L}
\end{eqnarray*}

If $N \ge t_2/L_1$, as $\alpha L N \ge 4N$, we can use~(\ref{attrape}) with $t\in[4N,\alpha LN]$; thus our choices for $\alpha,L$ and~(\ref{inclusion1}) ensure that on the event $\displaystyle A_1^{N}\cap (A_2^{t_2,N}\circ \theta_{-L_1N})  \cap A_3^N$, for every $ (x_0,t_0) \in B_N$
\begin{eqnarray*}
 \miniop{}{\cup}{4N\le s\le\alpha L N-t_0}  \xi^{x_0}_{s}\circ \theta_{t_0} \subset ((L_1+1+\alpha L)N)(1+\epsilon/2)A_{\mu} & \subset&  [-LN,LN]^d, \\
 \miniop{}{\cup}{0\le s\le\alpha 4N}  \xi^{x_0}_{s}\circ \theta_{t_0} \subset [-(4M+1)N,(4M+1)N]^d& \subset&  [-LN,LN]^d, \\
\xi^{x_0}_{\alpha L N-t_0}\circ \theta_{t_0} \subset  (L_1+1+\alpha L )N(1+\epsilon/2)A_{\mu}& \subset & x_0+(1+\epsilon)(\alpha L N-t_0)A_{\mu}.
\end{eqnarray*}
Finally, if $N \ge t_2/L_1$, 
$$A_1^{N}\cap (A_2^{t_2,N}\circ \theta_{-L_1N}) 
 \cap A_3^N \subset A^{\alpha,L,N,\epsilon},$$
and we conclude with (\ref{probab1}), (\ref{probab2}) and (\ref{probab3}).
\end{proof}


We first prove the existence of $C>0$ such that, with a large probability, the point $(0,0)$ can not give birth to more than $Ct$ generations before time $t$:
\begin{lemme}
\label{histoirelineaire}
There $A,B,C>0$ such that for every $\lambda\in [0,\lambda_{\max}]^{\Ed}$, for every $t,\ell\ge 0$:
$$
\P_\lambda \left(
\begin{array}{c}
 \exists (x,s)\in \Zd\times [0,t]  \text{ and an infection path from }(0,0) \\
\text{ to } (x,s)\text{ with more than }Ct+\ell\text{ horizontal edges }\end{array}
\right) \le A \exp(-B\ell).
$$
\end{lemme}

\begin{proof}
Let $\alpha>0$ be fixed. For every path $\gamma$ in $\Zd$ starting from $0$ and eventually self-intersecting, we set
$$X_{\gamma}=\1_{\{\gamma\text{ is the projection on $\Zd$ of an infection path starting from }(0,0)\}}e^{-\alpha t(\gamma)},$$ 
where  $t(\gamma)$ is the time when the extremity is infected after visiting successively the previous points. 
More formally, if the sequence of points in  $\gamma$ is $(0=x_0,\dots,x_n)$ and if we set $T_0=0$, and for $k\in\{0,\dots,n-1\}$, $$T_{k+1}=\inf\left\{t>T_k; \omega_{\{x_k,x_{k+1}\}}([T_k,t])=1\text{ and }\omega_{x_k}([T_k,t])=0\right\},$$ we have $t(\gamma)=T_n$.
The random variable $t(\gamma)$ is a stopping time (it is infinite if $\gamma$ is not the projection of an infection path).

Let $\gamma$ be a path in $\Zd$ starting from $0$ and let  $f$ be an edge at the extremity of~$\gamma$. If we denote by $\gamma.f$ the concatenation of $\gamma$ with $f$, the strong Markov property at time $t(\gamma)$ ensures that
$$\E_{\lambda} [X_{\gamma.f}|\mathcal{F}_{t(\gamma)}]\le  X_{\gamma} \frac{\lambda_{\max}}{\alpha+\lambda_{\max}}, \text{ and so } \E [X_{\gamma}]\le \left( \frac{\lambda_{\max}}{\alpha+\lambda_{\max}} \right)^{|\gamma|}.$$
Now,
\begin{eqnarray*}
 && \P_\lambda \left(
\begin{array}{c}
 \exists (x,s)\in \Zd\times [0,t]  \text{ is an infection path from }(0,0) \\
\text{ to } (x,s)\text{ with more than }Ct+\ell\text{ horizontal edges}\end{array}
\right) \\
& = & \P_\lambda \left(\miniop{}{\cup}{\gamma:|\gamma|\ge Ct+\ell} \{X_{\gamma}\ge e^{-\alpha t}\} \right) \\
& \le & e^{\alpha t} \miniop{}{\sum}{\gamma:|\gamma|\ge Ct+\ell} \left(\frac{\lambda_{\max}}{\alpha+\lambda_{\max}}\right)^{|\gamma|}  \le e^{\alpha t} \miniop{}{\sum}{n\ge Ct+\ell} \left(\frac{2d\lambda_{\max}}{\alpha+\lambda_{\max}}\right)^{n} .
\end{eqnarray*}
To conclude, we take $\alpha=2d\lambda_{\max}$, and then  $C$ such that $(\frac{2d}{2d+1})^{C}=e^{-\alpha }$.
\end{proof}

\begin{proof}[Proof of Theorem~\ref{GDdessous}]
Let $\epsilon>0$ and $t>0$ be fixed. 
Obviously
\begin{eqnarray}
\label{majgross}
\P(\xi^0_t \not \subset (1+\epsilon)tA_{\mu})&\le &\P(\xi^0_t \not \subset (1+\epsilon)tA_{\mu},\xi^0_t  \subset [-Mt,Mt]^d)\\&&+\P(\xi^0_t \not \subset [-Mt,Mt]^d)\nonumber.
\end{eqnarray}
The second term is controlled by equation~\eqref{richard}

Assume that $\xi^0_t \not \subset (1+\epsilon)tA_{\mu}$: let $x \in \xi^0_t$ be such that $\mu(x) \ge (1+\epsilon)t$, $\|x\|_{\infty}\le Mt$ and let $\gamma$ be an infection path from $(0,0)$ to $(x,t)$.
With Lemma~\ref{histoirelineaire}, we choose $C>1,A_2,B_2>0$  such that for every $t \ge 0$,
\begin{equation}
\label{pastropdaretes}
 \P \left( \begin{array}{c}
\text{there exists an infection path from } (0,0) \text{ to } \Zd\times\{t\} \\
\text{with more than } Ct \text{ horizontal edges}
\end{array}
\right) \le A_2 \exp(-B_2t).
\end{equation}
 With the last estimate, we can assume that $\gamma$ has less than $Ct$ horizontal edges. 

We take $0<\alpha<1$ and $L=L(\alpha,\epsilon)$ large enough to apply Lemma~\ref{catendvers12} and such that 
\begin{equation}
\label{choiceL}
\frac{4C_\mu^+C}{\alpha L-1}\le \frac{\epsilon}3, \quad \alpha L\ge 2\quad \text{ and } L\ge 3.
\end{equation}
We fix an integer $N$ and we cut the space-time  $\Zd\times \R_+$ into space-time boxes:
$$\forall k \in \Zd \quad \forall n \in \N \quad B_N(k,n)=(2Nk+[-N,N]^d)\times(2Nn+[0,2N]).$$
We associate to the path $\gamma$ a finite sequence $\Gamma=(k_i,n_i,a_i,t_i)_{0 \le i \le {\ell}}$, where the $(k_i,n_i)\in\Z^d\times \N$ are the coordinates of  space-time boxes  and the  $(a_i,t_i)$ are points in $\Zd \times \R_+$ in the following manner:
\begin{itemize}
\item $k_0=0$, $n_0=0$, $a_0=0$ and $t_0=0$: $B_N(k_0,n_0)$ is the box containing the starting point  $(a_0,t_0)=(0,0)$ of the path $\gamma$.
\item Assume we have chosen $(k_i,n_i,a_i,t_i)$, where $(a_i,t_i)$ is a point in $\gamma$ and $(k_i,n_i)$ are the coordinates of the space-time box containing $(a_i,t_i)$. To the box $B_N(k_i,n_i)$, we add the larger box $(2Nk_i+[-LN,LN]^d)\times (2Nn_i+[0,\alpha LN])$, we take for $(a_{i+1},t_{i+1})$ the first point -- if it exists -- along $\gamma$ after $(a_i,t_i)$ to be outside this large box, and we take for $(k_{i+1},n_{i+1})$ the coordinates of the space-time box that containing $(a_{i+1},t_{i+1})$. Otherwise, we stop the process.
\end{itemize}
The idea is to extract from the path a sequence of large portions,  \ie  the portions of $\gamma$ between $(a_i,t_i)$ and $(a_{i+1},t_{i+1})$.
We have the following estimates:
\begin{eqnarray}
&&\forall i\in[0..{\ell}-1] \quad \|a_{i+1}-a_i\|_\infty\le (L+1)N \text{ and } \|a_l-x\|_\infty \le (L+1)N, \label{absolument} \\
&&\forall i\in[0..{\ell}-1] \quad 0 \le t_{i+1}-t_i\le \alpha LN \text{ and } 0 \le t-t_l \le \alpha LN,  \label{cequetuveux}\\
&& 1\le  {\ell} \le \frac{Ct}{(L-1)N}+\frac{t}{(\alpha L-1)N}+2 \le\frac{2Ct}{(\alpha L-1)N}+2 . \label{majoratl}
\end{eqnarray}
The two first estimates just say that -- spatially for~\eqref{absolument} and in time for~\eqref{cequetuveux}-- the point $(a_{i+1},t_{i+1})$ remains in the large box centered around $B_N(k_{i},n_{i})$, which contains $(a_i,t_i)$.
Now consider the third estimate.
We note that a path can get out of a large box either with its time coordinate -- and the number of such exits is smaller than  $\frac{t}{(\alpha L-1) N}+1$ --, or by the space coordinate -- , and the number of such exits is smaller than  $\frac{Ct}{(L-1)N}+1$. The last inequality comes from $C>1$ and $\alpha<1$.

To ensure that the space coordinates of the boxes associated to the path are all distinct, we extract a subsequence $\overline{\Gamma}=(k_{\varphi(i)})_{0 \le i \le \overline{{\ell}}}$ with the loop-removal process described by Grimmett--Kesten~\cite{grimmett-kesten}:
\begin{itemize}
\item $\varphi(0)=\max\{j\ge 0: \; \forall i \in [0..j] \; k_i=0\}$;
\item Assume we chose $\varphi(i)$, then we take, if it is possible,
\begin{eqnarray*}
j_0(i) & = & \inf\{j >\varphi(i): k_j \neq k_{\varphi(i)}\}, \\
\varphi(i+1) & = & \max\{j\ge j_0(i): \;  k_j=k_{j_0(i)}\}.
\end{eqnarray*}
and we stop the extraction process otherwise.
\end{itemize}
Then, as in~\cite{grimmett-kesten}
\begin{eqnarray*}
&&\|a_{\varphi(\overline{{\ell}})}-x\|_\infty \le (L+1)N, \\
&& 0 \le t-t_{\varphi(\overline{{\ell}})} \le \alpha LN, \\
&& \forall i \in [0..\overline{{\ell}}-1] \quad \|a_{\varphi(i)+1}-a_{\varphi(i+1)}\|_{\infty} \le 2N,\\
&& \forall i \in [0..\overline{{\ell}}-1] \quad |t_{\varphi(i)+1}-t_{\varphi(i+1)}| \le 2N.
\end{eqnarray*}
Moreover, the upper bound~(\ref{majoratl}) for ${\ell}$ ensures that
\begin{equation}
\label{majoratlbar}
1\le \overline{{\ell}}\le {\ell} \le \frac{2Ct}{(\alpha L-1)N}+2.
\end{equation}
On the other hand, as $\displaystyle \mu(x)-\mu(a_{\varphi(\overline{{\ell}})}-x) \le \mu(a_{\varphi(\overline{l})})$, we have with~(\ref{majoratlbar}):
\begin{eqnarray*}
(1+\epsilon)t -C^+_{\mu}(L+1)N
& \le &\mu \left( \sum_{i=0}^{\overline{{\ell}}-1}a_{\varphi(i+1)}- a_{\varphi(i)}\right) \\
& \le &\sum_{i=0}^{\overline{{\ell}}-1} \mu(a_{\varphi(i+1)}-a_{\varphi(i)+1}) +\sum_{i=0}^{\overline{{\ell}}-1} \mu(a_{\varphi(i)+1}-a_{\varphi(i)}) \\
& \le &2N C_\mu^+\overline{{\ell}} +\sum_{i=0}^{\overline{{\ell}}-1} \mu(a_{\varphi(i)+1}-a_{\varphi(i)})\\
& \le &2NC_\mu^+\left(\frac{2Ct}{(\alpha L-1)N}+2\right) +\sum_{i=0}^{\overline{{\ell}}-1} \mu(a_{\varphi(i)+1}-a_{\varphi(i)}).
\end{eqnarray*}
This ensures, with the choice~(\ref{choiceL}) we made for $\alpha,L$, that
\begin{equation}
\label{senex}
\sum_{i=0}^{\overline{{\ell}}-1} \mu(a_{\varphi(i)+1}-a_{\varphi(i)}) \ge (1+2\epsilon/3)t -2C^+_{\mu}(L+1)N.
\end{equation}
In other words, even after the extraction process, the sum of the lengths of the crossings remains of order  $(1+2\epsilon/3)t$.

Let $k \in \Zd$ and $n \in \N$. We say now that $B_N(k,n)$ is good if 
$$\text{the event }A^{\alpha,L,N,\epsilon/3} \circ T_{2kN} \circ \theta_{2nN} \text{ occurs},$$ 
and bad otherwise.
If $B_N(k_{\varphi(i)},n_{\varphi(i)})$ is good, then the path exits the corresponding large box by the time coordinate, and thus $\mu(a_{\varphi(i)+1}-a_{\varphi(i)})\le (1+\epsilon/3)(t_{\varphi(i)+1}-t_{\varphi(i)})$; this ensures that
\begin{eqnarray*}
\mu \left(\sum_{i: \; B_N(k_{\varphi(i)},n_{\varphi(i)}) \text{ good}}(a_{{\varphi(i)}+1}-a_{\varphi(i)}) \right) 
& \le &  \sum_{i: \; B_N(k_{\varphi(i)},n_{\varphi(i)})\text{ good}} \mu(a_{{\varphi(i)}+1}-a_{\varphi(i)}) \\
& \le & (1+\frac{\epsilon}3)\sum_{i:\; B_N(k_{\varphi(i)},n_{\varphi(i)}) \text{ good}}(t_{{\varphi(i)}+1}-t_{\varphi(i)}) \\
& \le & (1+\frac{\epsilon}3) t.
\end{eqnarray*}
With~\eqref{senex}, it implies that
$$ \sum_{i:  \; B_N(k_{\varphi(i)},n_{\varphi(i)})  \text{ bad}}\mu(a_{{\varphi(i)}+1}-a_{\varphi(i)})  \ge \frac{\epsilon}3t-2C^+_{\mu}(L+1)N,$$ and then, with~\eqref{absolument}
$$\Card{\{i:  \; B_N(k_{\varphi(i)},n_{\varphi(i)})  \text{ bad}\}} \ge \frac{\epsilon t}{3C^+_\mu (L+1)N}-2.$$
In other words, if $t>0$, if $x$ is such that $\mu(x) \ge (1+\epsilon)t$, if there exists an infection path $\gamma$ from $(0,0)$ to $(x,t)$ with less than $Ct$ horizontal edges, the associated sequence $\overline{\Gamma}$ has a number of bad boxes proportional to $t$.

Note that Lemma~\ref{catendvers12} says that for any deterministic family $n=(n_k)_{k \in \Zd} \in \N^{\Zd}$,
the field $(\eta^n_k)_{k \in \Zd}$, defined by 
$\eta^n_{k}=\1_{\{ B_N(k,n_{k})  \text{ good}\}}$ is locally dependent and that
$$\lim_{N \to +\infty} \P(B_N(0,0) \text{ good})=1.$$
By the extraction process, the spatial coordinates of the boxes in $\overline{\Gamma}$ are all distinct. With the comparison theorem by Liggett--Schonmann--Stacey~\cite{LSS}, we can, for any $p_1<1$, take $N$  large enough to ensure that for any family  $n=(n_k)_{k \in \Zd} \in \N^{\Zd}$, the law of the field $(\eta^n_k)_{k \in \Zd}$ under $\P$ stochastically dominates a product on $\Zd$ of Bernoulli laws with parameter $p_1$. Thus, if $x$ is such that $\mu(x) \ge (1+\epsilon)t$, then
\begin{eqnarray*}
&  &\P \left(\begin{array}{c}
\text{there exists an infection path $\gamma$ from $(0,0)$ to $(x,t)$}\\
\text{with less than $Ct$ horizontal edges and such that $\overline{\Gamma}=\overline{\Gamma}(\gamma)$ has}\\
\text{at least $\frac{\epsilon t}{3C^+_\mu (L+1)N}-2$ bad boxes}
         \end{array} 
\right) \\
& \le & \sum_{{\ell}=1}^{\frac{2Ct}{(\alpha L-1)N}+2} \sum_{\Card{\overline{\Gamma}}={\ell}}2^{\ell}(1-p_1)^{\frac{\epsilon t}{3C^+_\mu (L+1)N}-1}
\\ &=& (1-p_1)^{\frac{\epsilon t}{3C^+_\mu (L+1)N}-1}\sum_{{\ell}=1}^{\frac{2Ct}{(\alpha L-1)N}+2}2^{\ell}\textrm{Card}{\{\overline{\Gamma};\Card{\overline{\Gamma}}=\ell\}}
\end{eqnarray*}
A classical counting argument gives the existence of a constant $K=K(d,\alpha,L)$ independent of $N$ such that
$$\forall \ell\ge 1\quad \textrm{Card}{\{\overline{\Gamma};\Card{\overline{\Gamma}}=\ell\}}\le K^{\ell}.$$
We get then an upper bound for our probability of the form
$$A \frac{t}N \left( (1-p_1)^{\frac{\epsilon}{3C_\mu^+(L+1)}} (2K)^{\frac{2C}{\alpha L-1}}\right)^{t/N},$$ which leads to a bound of the form $A_3\exp(-B_3 t)$
as soon as  $p_1$ is close enough to~$1$. 
Summing over all $x\in [-Mt,Mt]^d$, we have again an exponential bound.
With this last upper bound,~(\ref{majgross}) and (\ref{pastropdaretes}), we end the proof of Theorem~\ref{GDdessous}.
\end{proof}

\begin{proof}[Proof of Theorem~\ref{dessouscchouette}]
We first prove there exist $A,B>0$ such that
\begin{equation}
\label{GDTgrand}
\forall T > 0\quad \P (\exists t\ge T\quad \xi^0_t \not \subset (1+\epsilon)tA_\mu)\le A\exp(-BT).
\end{equation}
Indeed,
\begin{eqnarray*}
& & \P (\exists t\ge T\quad \xi^0_t \not \subset (1+\epsilon)tA_\mu))\\ & \le & \P(\exists n\in\N \quad \xi^0_{T+n} \not \subset (1+\epsilon/2)(T+n)A_\mu)\\
& & + \P(\exists n\in\N \quad \exists t\in [0,1]\quad \xi^0_{T+n} \subset (1+\epsilon/2)(T+n)A_\mu, \xi^0_{T+n+t} \not \subset ((1+\epsilon)(T+n)A_\mu) \\
& \le & \sum_{n\ge 0} \P(\xi^0_{T+n} \not \subset (1+\epsilon/2)(T+n)A_\mu)\\
& & +\sum_{n\ge 0}\P(\exists t\in [0,1]\quad \xi^0_{T+n} \subset (1+\epsilon/2)(T+n)A_\mu, \xi^0_{T+n+t} \not \subset ((1+\epsilon)(T+n)A_\mu).
\end{eqnarray*}
The first sum can be controlled with Theorem~\ref{GDdessous}.
For the second sum, the Markov property gives for any $\lambda\in\Lambda$, 
\begin{eqnarray*}
& &\P_{\lambda}(\exists t\in [0,1]\quad \xi^0_{T+n} \subset (1+\epsilon/2)(T+n)A_\mu, \; \xi^0_{T+n+t} \not \subset ((1+\epsilon)(T+n)A_\mu)\\ 
&  \le &\sum_{x\in (1+\epsilon/2)(T+n)A_\mu}\P_{x.\lambda}(\exists t\in [0,1]\quad \xi^0_{t} \not \subset (\epsilon/2)(T+n)A_\mu)\\
&  \le &\Card{(1+\epsilon/2)(T+n)A_\mu}\P(H_1^0 \not \subset (\epsilon/2)(T+n)A_\mu)\le A \exp(-B (T+n)),
\end{eqnarray*}
where the last upper bound comes from a comparison with the Richardson model. 
We conclude the proof of~(\ref{GDTgrand}) by integrating with respect to $\lambda$.

Let us prove now the existence of $A,B>0$ such that
\begin{equation}
\label{GDHpetit}
\forall r>0\quad \P (H^0_r \not\subset(1+\epsilon)r A_{\mu})\le A\exp(-Br).
\end{equation} 
With~\eqref{richard}, we can find $A_1,B_1>0$ and $c<1$ such that
$\P (H^0_{cr} \not\subset r A_{\mu})\le A_1\exp(-B_1r).$
Now,
\begin{eqnarray*}
\P (H^0_r \not\subset(1+\epsilon)r A_\mu) 
& \le & \P (H^0_{cr} \not\subset r A_\mu)+\P(\exists t \in( cr,r)  \quad \xi^0_t \not\subset (1+\epsilon)r A_\mu) \\
& \le & A_1\exp(-B_1r) +\P(\exists t \ge cr \quad \xi^0_t \not\subset (1+\epsilon)t A_\mu),
\end{eqnarray*}
and we conclude the proof of~\eqref{GDHpetit} with~(\ref{GDTgrand}).
To obtain~\eqref{decadix}, we just need to note that $t\mapsto H_t$ is non-decreasing.

Finally, for $x \in \Zd \backslash\{0\}$, 
$$\P(t(x) \le (1-\epsilon)\mu(x)) \le \P(H^0_{(1-\epsilon)\mu(x)} \not\subset \mu(x)A_\mu).$$
Applying~(\ref{GDHpetit}), we end the proof of~\eqref{defontenay}, and thus of Theorem~\ref{dessouscchouette}. 
\end{proof}
\section{About the order of the deviations}
\label{bonnevitesse}

By Theorems~\ref{theoGDUQ} and~\ref{dessouscchouette}, we have for $\nu$-almost every $\lambda$ and each $\epsilon>0$:

$$\miniop{}{\limsup}{x\to +\infty}\frac1{\|x\|}\log \Pbarre_{\lambda}\left(\frac{t(x)}{\mu(x)}\not\in [1-\epsilon,1+\epsilon]\right)<0.$$

To see that the exponential decrease in $\|x\|$ is optimal, we need too see that $\miniop{}{\liminf}{x\to +\infty}\frac1{\|x\|}\log \Pbarre_{\lambda}\left(\frac{t(x)}{\mu(x)}\not\in [1-\epsilon,1+\epsilon]\right)>-\infty.$


In fact, we will prove here that for every  $(s,t)$ with $0<s<t$, there exists a constant $\gamma>0$
such that for each $\lambda\in\Lambda$ and each $x\in\Zd$,
\begin{eqnarray*}
\P_\lambda(t(x)\in [s,t]\|x\|_1) & \ge & \exp(-\gamma\|x\|_1).
\end{eqnarray*}
\begin{proof}
Let $s,t$ with $0 <s <t$.
For each $u \in \Zd$ such that $\|u\|_1=1$, we define $T_u=\inf\{t \ge 0: \; \xi^0_t=\{u\},\quad \forall s\in [0,t)\quad \xi^0_s=\{0\}\}$. We are going to prove that
$$\exists \gamma>0 \quad \forall \lambda \in \Lambda \quad \forall u \in \Zd, \; \|u\|_1=1 \quad \P_\lambda(T_u \in [s,t]) \ge e^{-\gamma}.$$
In order to ensure that $T_u\in [s,t]$, it it sufficient to satisfy 
\begin{itemize}
\item The lifetime of the particle at $(0,0)$ is strictly between $(s+t)/2$ and $t$, which happens with probability  $e^{-(s+t)/2}-e^{-t}$ under $\P_\lambda$ ;
\item The first opening of the bond between $0$ and $u$ happens strictly between $s$ and $(s+t)/2$,  which happens with probability
$$\exp(-\lambda_{\{0,u\}} s)-\exp(-\lambda_{\{0,u\}} (s+t)/2) \ge \exp(-\lambda_{\max} s)(1-\exp(-\lambda_{\min}(t-s)/2))$$ under $\P_\lambda$;
\item There is no opening between time $0$ and time $t$, on the set $J$ constituted by the $4d-2$ bonds that are neighour of $0$ or  $u$ and differ from $\{0,u\}$, which happens under $\P_\lambda$ with probability
$$\prod_{j\in J} \exp(-\lambda_j t) \ge \exp(-(4d-2)\lambda_{\max} t);$$
\item There is no death at site $u$ between $0$ and $t$, which happens under $\P_\lambda$ with probability $e^{-t}$.
\end{itemize}
Then, using the independence of the Poisson processes, we get
\begin{eqnarray*}& &\P_{\lambda}(T_u \in [s,t])\\
& \ge & (e^{-(s+t)/2)}-e^{-t})e^{-t} e^{-(4d-2)\lambda_{\max} t} e^{-\lambda_{\max} s}(1-e^{-\lambda_{\min}(t-s)/2})=e^{-\gamma}.
\end{eqnarray*}
Moreover, $T_u$ is obviously a stopping time.
Then, applying the strong Markov property $\|x\|_1$ times, we get,
$$\P_{\lambda}(t(x)\in [s,t]\|x\|_1)\ge \exp(-\gamma\|x\|_1).$$
This gives the good speed for both upper and lower large deviations.
\end{proof}
Note that the order of the large deviations is the same for upper and lower deviations, as happens for the chemical distance in Bernoulli percolation (see Garet--Marchand~\cite{GM-large}).  Conversely, it is known that these orders may differ for first-passage percolation (see Kesten~\cite{kesten} and Chow--Zhang~\cite{chow-zhang}).



\def\refname{References}
\bibliographystyle{plain}


\end{document}